\documentclass[12pt,a4paper,openany]{article}
\usepackage{fancyhdr,mathrsfs,amssymb,amsmath,amsthm,latexsym}

\theoremstyle{plain}
\newtheorem*{theorem*}{Theorem}
\newtheorem{theorem}{Theorem}
\newtheorem{lemma}[theorem]{Lemma}

\newtheorem*{claim*}{Claim}

\newtheorem{conj}[theorem]{Conjecture}
\newtheorem{quest}[theorem]{Question}

\theoremstyle{definition}

\theoremstyle{remark}

\newcommand{\G}{\ensuremath{\mathcal{G}}}

\newcommand{\Z}{\ensuremath{\mathbb{Z}}}

\newcommand{\R}{\ensuremath{\mathcal{R}}}

\newcommand{\GG}{\ensuremath{\mathbb{G}}}

\usepackage{enumerate}
\usepackage[utf8]{inputenc}
\usepackage{pifont}
\usepackage{graphicx}
\usepackage{subcaption}
\usepackage{pdfpages}
\usepackage{wrapfig}
\usepackage{xcolor}
\usepackage[normalem]{ulem}
\usepackage{blindtext}
\usepackage{caption}
\usepackage{subcaption}
\usepackage{amsmath}
\usepackage{blindtext}
\usepackage{geometry}
\usepackage{fancyhdr}
\usepackage{tikz}
\usetikzlibrary{calc}
\usepackage{pgfplots}
\pgfplotsset{compat = newest}
\usepackage{mathtools}

\title{Transparent Rectangle Visibility Graphs}
\author{Chaipattana Juntarapomdach\thanks{\,Department of Mathematics and Computer Science, Faculty of Science, Chulalongkorn University, Bangkok 10330, Thailand; \texttt{6770025623@student.chula.ac.th}.}
	\and Teeradej Kittipassorn\thanks{\,Department of Mathematics and Computer Science, Faculty of Science, Chulalongkorn University, Bangkok 10330, Thailand; \texttt{teeradej.k@chula.ac.th}.}}

\begin{document}
\maketitle

\begin{abstract}
	A \emph{transparent rectangle visibility graph (TRVG)} is a graph whose vertices can be represented by a collection of non-overlapping rectangles in the plane whose sides are parallel to the axes such that two vertices are adjacent if and only if there is a horizontal or vertical line intersecting the interiors of their rectangles.
	
	We show that every threshold graph, tree, cycle, rectangular grid graph, triangular grid graph and hexagonal grid graph is a TRVG.
	We also obtain a maximum number of edges of a bipartite TRVG and characterize complete bipartite TRVGs.
	More precisely, a bipartite TRVG with $n$ vertices has at most $2n-2$ edges.
	The complete bipartite graph $K_{p,q}$ is a TRVG if and only if $\min\{p,q\} \le 2$ or $(p,q) \in \{(3,3), (3,4)\}$.
        We prove similar results for the torus.
	Moreover, we study whether powers of cycles and their complements are TRVGs.
\end{abstract}

\section{Introduction}
Consider a collection of non-overlapping rectangles in the plane where the sides are parallel to the axes.
Two rectangles $A$ and $B$ \emph{see each other} if there is a horizontal or vertical line segment intersecting the interiors of $A$ and $B$ such that it passes through no other rectangles.
This naturally induces a graph whose vertices are the rectangles and two vertices are adjacent if their rectangles see each other.
We call a graph which can be constructed in this way, a \emph{rectangle visibility graph}, or \emph{RVG}.
The first problem that arises is: which graphs are RVGs?

The study of RVGs was initiated in 1976 by Garey, Johnson and So~\cite{GJS} as a tool for designing printed circuit boards.
They considered only $1 \times 1$ squares in the plane whose vertices are lattice points.
In 1995, Dean and Hutchinson~\cite{DH} discovered that a complete bipartite graph $K_{p,q}$ where $p\le q$ is an RVG if and only if $p\leq 4$, and moreover, every bipartite RVG with $n$ vertices has at most $4n-12$ edges.
A few years later, in 1999, Hutchinson, Shermer, and Vince~\cite{HSV} proved that every RVG with $n$ vertices has at most $6n-20$ edges, and this bound is best possible for $n \geq 8$.
In 1996, Bose, Dean, Hutchinson and Shermer~\cite{BDHS} found new classes of RVGs including  graphs with maximum degree four and graphs that can be decomposed into two caterpillar forests.
Recently, in 2022, Caughman, Dunn, Laison, Neudauer, Starr~\cite{CDLN} studied the area, perimeter, height and width of the bounding box of an RVG.
Many variations of RVG have been investigated including bar visibility graphs~\cite{DHVM}, $3$-dimensional box visibility graphs~\cite{BEFH, DHM, FM, GL} and polygon visibility graphs~\cite{GDEL, LMT}.

In this paper, we consider what happens if rectangles $A$ and $B$ still see each other even when there are other rectangles blocking the line of sight.
That is, we suppose that every rectangle is transparent.
Then we call such a \emph{transparent rectangle visibility graph}, or \emph{TRVG}. Figure~\ref{fig:trvgandrvg} illustrates an example of a collection of rectangles representing different graphs as a TRVG and an RVG.
\begin{figure}[h]
	\centering
	\begin{tikzpicture}[scale=0.7]
		\filldraw[red!60!white, draw=black, very thick] (0,0) rectangle (1,1);
		\filldraw[red!60!white, draw=black, very thick] (1,0) rectangle (2,1);
		\filldraw[red!60!white, draw=black, very thick] (2,0) rectangle (3,1);
		
		\filldraw (0,-0.5) circle(0.1cm);
		\filldraw (0.5,-1.36) circle(0.1cm);
		\filldraw (-0.5,-1.36) circle(0.1cm);
		\draw[very thick] (0,-0.5) -- (-0.5,-1.36);
		\draw[very thick] (0.5,-1.36) -- (-0.5,-1.36);
		\draw[very thick] (0.5,-1.36) -- (0,-0.5);
		
		\filldraw (3,-0.5) circle(0.1cm);
		\filldraw (3.5,-1.36) circle(0.1cm);
		\filldraw (2.5,-1.36) circle(0.1cm);
		\draw[very thick] (3,-0.5) -- (2.5,-1.36);
		\draw[very thick] (3.5,-1.36) -- (3,-0.5);
		\draw (0,-1.8) node[rectangle] {TRVG};
		\draw (3,-1.8) node[rectangle] {RVG};
	\end{tikzpicture}
	\caption{A difference between a TRVG and an RVG}
	\label{fig:trvgandrvg}
\end{figure}
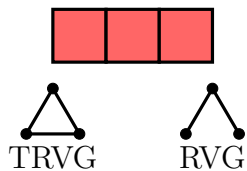

Note that every induced subgraph of a TRVG is also a TRVG by deleting all rectangles corresponding to the vertices outside the induced subgraph.

Our first main result proves that some classes of graphs are TRVGs.

\begin{theorem}\label{thm:istrvg}
	The following graphs are TRVGs:
	\begin{enumerate}[(i)]
		\item a threshold graph,
            \item a tree,
		\item a cycle,
		\item an infinite rectangular grid graph,
		\item an infinite triangular grid graph and
		\item an infinite hexagonal grid graph.
	\end{enumerate}
\end{theorem}

The following two main results for bipartite TRVGs are analogous to those of Dean and Hutchinson~\cite{DH} for bipartite RVGs.
Consequently, we discover some bipartite non-TRVGs.

\begin{theorem}\label{thm:bipartTRVG}
    Every bipartite TRVG on $n$ vertices has at most $2n-2$ edges. Moreover, this is best possible for $n\ge7$.
\end{theorem}

Applying Theorem~\ref{thm:bipartTRVG}, we can classify complete bipartite graphs.

\begin{theorem}\label{thm:compbip}
    For $p\le q$, $K_{p,q}$ is a TRVG if and only if $p\le 2$ or $(p,q)\in\{(3,3),(3,4)\}$.
\end{theorem}

We see that $K_{4,4}$ and $K_{3,5}$ are non-TRVGs with respect to the plane. However, they are TRVG with respect to the torus. The same ideas of the proofs of Theorems~\ref{thm:bipartTRVG} and~\ref{thm:compbip} can be applied to the case of the torus.

\begin{theorem}\label{thm:bipartTRVGtorus}
    Every bipartite TRVG with respect to the torus on $n$ vertices has at most $2n$ edges. Moreover, this is best possible for $n\ge8$.
\end{theorem}

\begin{theorem}\label{thm:compbiptorus}
    For $p\le q$, $K_{p,q}$ is a TRVG with respect to the torus if and only if $p\le 2$ or $(p,q)\in\{(3,3),(3,4),(3,5),(3,6),(4,4)\}$.
\end{theorem}

Furthermore, we consider a special case of regular graphs. Vertices $v_1,\dots,v_n$ form a cycle $C_n$ if $v_i$ and $v_j$ are adjacent when $|i-j|\equiv 1 \pmod{n}$.
Generalizing this, they form the \emph{$a^{th}$ power of $C_n$} if $v_i$ and $v_j$ are adjacent when $|i-j|\equiv k \pmod{n}$ for some $k\le a$. We denote this graph by $C^a_n$ and its complement by $D^a_n$


\begin{theorem}\label{thm:poc}
    \begin{enumerate}[(i)]
        \item Every power of a cycle is a TRVG.
        \item $D^1_n$ is a TRVG.
        \item $D^a_n$ with $n\ge 2a+8$ is a non-TRVG when $a\ge 3$.
        \item $D^a_n$ with $n\le 2a+4$ is a TRVG when $a\ge 2$.
    \end{enumerate}
\end{theorem}

Let us introduce a definition that will be used throughout this paper.
Given a collection of rectangles in the plane representing a graph $G$ and a subset of vertices $U \subseteq V(G)$, a \emph{bounding box of $U$} is a box which covers all rectangles representing $U$ and intersects no other rectangles.

This paper is organized as follows.
Sections~\ref{secthreshold},~\ref{sectree} and~\ref{secgrid} focus on the proof of Theorem~\ref{thm:istrvg}.
In Section~\ref{secbipartite}, we study bipartite graphs and prove Theorems~\ref{thm:bipartTRVG},~\ref{thm:compbip},~\ref{thm:bipartTRVGtorus} and~\ref{thm:compbiptorus}.
We prove Theorem~\ref{thm:poc} regarding powers of cycles in Section~\ref{secpoc}.
Finally, we leave some open problems in Section~\ref{sec:conclude}.

\section{Threshold Graph}\label{secthreshold}

In this section, we will prove Theorem~\ref{thm:istrvg}$(i)$, i.e. every threshold graph is a TRVG.

As one of the fundamental classes of graphs, threshold graphs have been defined in many ways (see~\cite{Pnoey}). A graph $G$ is a \emph{threshold graph} if we can assign a real number $r_v$ to each vertex $v$ and there is a real number $\theta$ such that for any vertex subset $U$ of $G$, $\sum_{v\in U}r_v\le\theta$ if and only if $U$ is independent in $G$. Equivalently, it is a graph that can be constructed by repeating these two operations:
	
    \begin{itemize}
\item Add an isolated vertex to the graph.

\item Add a \emph{universal vertex} to the graph, i.e. a vertex connected to all others.
    \end{itemize}

\begin{proof}[Proof of Theorem~\ref{thm:istrvg}$(i)$]
We will proceed by induction on the number of vertices. Clearly, a one-vertex graph is a TRVG.

For the inductive step, let $G$ be a threshold graph and let $v$ be the last vertex added to $G$ in the threshold graph construction.
We see that $G-v$ is also a threshold graph, and so $G-v$ is a TRVG by the induction hypothesis.
Thus, $G-v$ can be represented by a collection of rectangles in the plane.

Considering whether $v$ is isolated or universal, we can construct a rectangle for $v$ that intersects no line of sight or covers a side of the bounding box of $G-v$ respectively (see Figure~\ref{fig:thresproof}).
\begin{figure}[h]
    \centering
\begin{tikzpicture}[scale=0.6]
    \draw[very thick] (-0.4,-0.4) rectangle (4.4,4.4);
    \draw[step=0.5cm,gray,dashed] (-0.4,-0.4) grid (4.4,4.4);
    \filldraw[blue!60!white, draw=black, very thick] (0,0) rectangle (3,3.5);
    \draw (1.5,2.4) node[scale=0.8] {Bounding};
    \draw (1.5,1.75) node[scale=0.8] {box of};
    \draw (1.5,1.1) node[scale=0.8] {$G-v$};
    \filldraw[red!60!white, draw=black,very thick] (3,0) rectangle (3.5,3.5);
    \draw (3.25,1.75) node[scale=0.8] {$v$};
    \filldraw[red!60!white, draw=black,very thick] (3.5,3.5) rectangle (4,4);
    \draw (3.75,3.75) node[scale=0.8] {$v$};
 
\end{tikzpicture}
\caption{An addition of an isolated or universal vertex $v$}
\label{fig:thresproof}
\end{figure}
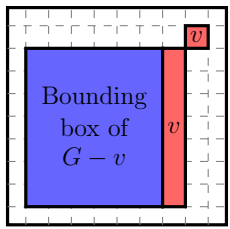
\end{proof}

Consequently, the empty graph $E_n$ and the complete graph $K_n$ are TRVGs since they are threshold graphs.
Moreover, every graph formed by adding isolated vertices or universal vertices to a TRVG is also a TRVG.

Recall that Hutchinson, Shermer, and Vince~\cite{HSV} proved that every RVG with $n$ vertices has at most $6n-20$ edges when $n \ge 5$. It follows that $K_n$ is not an RVG for $n \ge 9$.
Thus, $K_n$ is an example of a TRVG which is not an RVG.
Furthermore, a TRVG cannot have nontrivial bounds for the number of edges as an RVG since $E_n$ and $K_n$ are TRVGs.

\section{Tree}\label{sectree}

We will start with an important lemma which will be used to prove Theorem~\ref{thm:istrvg}$(ii)$, i.e. every tree is a TRVG.

The \emph{union} and the intersection of two graphs $G_1 = (V_1, E_1)$ and $G_2 = (V_2, E_2)$ are defined as the graph $G_1 \cup G_2 = (V_1 \cup V_2, E_1 \cup E_2)$ and $G_1 \cap G_2 = (V_1 \cap V_2, E_1 \cap E_2)$ respectively.

\begin{lemma}\label{lemma:disjointv}
    Let $G_1$ and $G_2$ be TRVGs with $V(G_1) \cap V(G_2)=\{v\}$.
If there is a representation of $G_i$ for each $i \in \{1,2\}$ such that, in one direction, the rectangle for $v$ sees no other rectangles then $G_1 \cup G_2$ is a TRVG.
\end{lemma}

\begin{proof}
Without loss of generality, consider a collection of rectangles in the plane representing $G_1$ such that the rectangle for $v$ sees no other rectangles horizontally and has size $1 \times 1$ since rotating the representation by $\frac{\pi}{2}$ and rescaling the representation in any direction do not affect the graph.
Similarly, consider a collection of rectangles in the plane representing $G_2$ such that the rectangle for $v$ sees no other rectangles vertically and has size $1 \times 1$.

We union the representation of $G_1$ and $G_2$ so that the rectangles for $v$ coincide.
Since the rectangle for $v$ sees no other rectangles of $G_1$ horizontally, each rectangle for $G_1$ is either above or below $v$.
Observe that an upward translation of the rectangles above $v$ and a downward translation of the rectangles below $v$ do not affect the graph $G_1$.
Similarly, a left shifting of the rectangles to the left of $v$ and a right shifting of the rectangles to the right of $v$ also do not affect the graph $G_2$.
Thus, we can translate other rectangles so that those from $G_1$ do not see those from $G_2$.
Hence, $G_1 \cup G_2$ is a TRVG.
\end{proof}
We are ready to prove Theorem~\ref{thm:istrvg}$(ii)$.

\begin{proof}[Proof of Theorem~\ref{thm:istrvg}$(ii)$]
Let $G$ be a tree with vertices $v_1, v_2, \dots, v_n$ such that $v_1$ is the root and the vertices are ordered by their distances to the root from small to large, i.e. $d(v_1,v_i)\le d(v_1,v_j)$ for $i<j$.
For a vertex $v_i$, let $G_i$ be the subtree consisting of the edges incident to $v_i$ going away from the root.
Then $G=G_1\cup G_2\cup\dots\cup G_n$ is a decomposition of $G$ into stars (see Figure~\ref{fig:treecompose}).
Let $\GG_i = G_1\cup G_2\cup\dots\cup G_i$.
We will prove that $\GG_i$ is a TRVG by induction on $i$.

Clearly, $\GG_1=G_1$ is a TRVG for the basis step.
For the inductive step, suppose that $\GG_{k-1}$ is a TRVG.
Since $d_{\GG_{k-1}}(v_k)=1$, the rectangle for $v_k$ sees rectangles of $\GG_{k-1}$ from one direction only.
We can represent the star $G_k$ such that the rectangle for $v_{k}$ sees other rectangles in $G_k$ from one direction only.
Since $\GG_{k-1} \cap G_k = \{v_k\}$, $\GG_k$ is also a TRVG by Lemma~\ref{lemma:disjointv}.
Thus, every tree is a TRVG.
\end{proof}
	
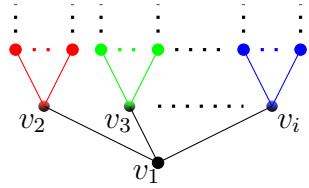
\begin{figure}
    \centering
    \begin{tikzpicture}[scale=0.75]
	\filldraw[black] (0,0) circle (0.1cm);
	\shade[bottom color=black,top color=red!60!white] (-2,1) circle (0.1cm);
	\shade[bottom color=black,top color=green!60!white] (-0.5,1) circle (0.1cm);
	\shade[bottom color=black,top color=blue!60!white] (2,1) circle (0.1cm);
	\filldraw[red] (-2.5,2) circle (0.1cm);
	\filldraw[red] (-1.5,2) circle (0.1cm);
	\filldraw[green] (-1,2) circle (0.1cm);
	\filldraw[green] (0,2) circle (0.1cm);
	\filldraw[blue] (2.5,2) circle (0.1cm);
	\filldraw[blue] (1.5,2) circle (0.1cm);
	\draw (0,0) -- (-2,1);
	\draw (0,0) -- (-0.5,1);
	\draw (0,0) -- (2,1);
	\draw[red] (-2,1) -- (-2.5,2);
	\draw[red] (-2,1) -- (-1.5,2);
	\draw[green] (-0.5,1) -- (-1,2);
	\draw[green] (-0.5,1) -- (0,2);
	\draw[blue] (2,1) -- (1.5,2);
	\draw[blue] (2,1) -- (2.5,2);
	\draw[loosely dotted, very thick] (0,1) -- (1.5,1);
	\draw[loosely dotted, very thick, red] (-2.2,2) -- (-1.8,2);
	\draw[loosely dotted, very thick, green] (-0.7,2) -- (-0.3,2);
	\draw[loosely dotted, very thick, blue] (1.8,2) -- (2.2,2);
	\draw[loosely dotted, very thick] (0.3,2) -- (1.2,2);
	\draw[loosely dotted, very thick] (-2.5,2.3) -- (-2.5,2.8);
	\draw[loosely dotted, very thick] (-1.5,2.3) -- (-1.5,2.8);
	\draw[loosely dotted, very thick] (-1,2.3) -- (-1,2.8);
	\draw[loosely dotted, very thick] (0,2.3) -- (0,2.8);
	\draw[loosely dotted, very thick] (1.5,2.3) -- (1.5,2.8);
	\draw[loosely dotted, very thick] (2.5,2.3) -- (2.5,2.8);
	\draw (-0.2,-0.2) node[rectangle] {$v_1$};
	\draw (-2.2,0.7) node[rectangle] {$v_2$};
	\draw (-0.7,0.7) node[rectangle] {$v_3$};
	\draw (2.3,0.7) node[rectangle] {$v_i$};
    \end{tikzpicture}
    \caption{$G_1$: Black, $G_2$: Red, $G_3$: Green, $G_i$: Blue}
    \label{fig:treecompose}
\end{figure}	

We follow the steps in the proofs of Theorem~\ref{thm:istrvg}$(ii)$ and Lemma~\ref{lemma:disjointv} to construct a representation of a tree as follows. 


	
	

Select any vertex $v$ from the tree to be the root and place its rectangle in the plane. Let $F_k$ be the set of vertices whose distance from $v$ is $k$.
First we place the rectangle for each vertex in $F_1$ on the plane inside the horizontal vision of the rectangle for $v$ but outside some bounding box of $\{v\}$.
Then we place the rectangle for each vertex in $F_2$ on the plane inside the vertical vision of the rectangle for its neighbor in $F_1$ but outside some bounding box of $\{v\}\cup F_1$.
Continue the process by placing the rectangle for each vertex in $F_k$ on the plane inside the horizontal vision if $k$ is odd or the vertical vision if $k$ is even of the rectangle for its neighbor in $F_{k-1}$ but outside some bounding box of $\{v\}\cup F_1\cup F_2\cup\dots\cup F_{k-1}$ (see Figure~\ref{fig:treeexample1}).

\begin{figure}[h]
	\centering
	\begin{tikzpicture}[scale=0.5]
		\filldraw[black] (0,0) circle (0.1cm);
		
		\draw (0,0) -- (-3,1);
		\filldraw[red] (-3,1) circle (0.1cm);
		\draw (0,0) -- (0,1);
		\filldraw[red] (0,1) circle (0.1cm);
		\draw (0,0) -- (2,1);
		\filldraw[red] (2,1) circle (0.1cm);
		
		\draw (-3,1) -- (-4,2);
		\filldraw[green] (-4,2) circle (0.1cm);
		\draw (-3,1) -- (-2,2);
		\filldraw[green] (-2,2) circle (0.1cm);
		\draw (0,1) -- (0,2);
		\filldraw[green] (0,2) circle (0.1cm);
		\draw (0,1) -- (1,2);
		\filldraw[green] (1,2) circle (0.1cm);
		\draw (2,1) -- (2,2);
		\filldraw[green] (2,2) circle (0.1cm);
		\draw (2,1) -- (3,2);
		\filldraw[green] (3,2) circle (0.1cm);
		
		\draw (-2,2) -- (-3,3);
		\filldraw[blue] (-3,3) circle (0.1cm);
		\draw (-2,2) -- (-2,3);
		\filldraw[blue] (-2,3) circle (0.1cm);
		\draw (-2,2) -- (-1,3);
		\filldraw[blue] (-1,3) circle (0.1cm);
		\draw (1,2) -- (1,3);
		\filldraw[blue] (1,3) circle (0.1cm);
		\draw (3,2) -- (3,3);
		\filldraw[blue] (3,3) circle (0.1cm);
		
		\draw (1,3) -- (0,4);
		\filldraw[yellow] (0,4) circle (0.1cm);
		\draw (1,3) -- (1,4);
		\filldraw[yellow] (1,4) circle (0.1cm);
		\draw (1,3) -- (2,4);
		\filldraw[yellow] (2,4) circle (0.1cm);
		\draw (3,3) -- (3,4);
		\filldraw[yellow] (3,4) circle (0.1cm);
		
		\draw (2,4) -- (2,5);
		\filldraw[purple] (2,5) circle (0.1cm);
		
		\draw[very thick] (5.05,-0.95) rectangle (12.45,6.45);
		\draw[step=0.5cm,gray] (5.05,-0.95) grid (12.45,6.45);
		\filldraw[black!60!white, draw=black, very thick] (8.5,3) rectangle (9,4.5);
		\filldraw[red!60!white, draw=black, very thick] (7.5,4) rectangle (8.5,4.5);
		\filldraw[red!60!white, draw=black, very thick] (6.5,3.5) rectangle (7.5,4);
		\filldraw[red!60!white, draw=black, very thick] (5.5,3) rectangle (6.5,3.5);
		\filldraw[green!60!white, draw=black, very thick] (5.5,2.5) rectangle (6,3);
		\filldraw[green!60!white, draw=black, very thick] (6,2) rectangle (6.5,2.5);
		\filldraw[green!60!white, draw=black, very thick] (6.5,1.5) rectangle (7,2);
		\filldraw[green!60!white, draw=black, very thick] (7,1) rectangle (7.5,1.5);
		\filldraw[green!60!white, draw=black, very thick] (7.5,-0.5) rectangle (8,1);
		\filldraw[green!60!white, draw=black, very thick] (8,-0.95) rectangle (8.5,-0.5);
		\filldraw[blue!60!white, draw=black, very thick] (9,-0.5) rectangle (9.5,0);
		\filldraw[blue!60!white, draw=black, very thick] (9.5,0) rectangle (10,0.5);
		\filldraw[blue!60!white, draw=black, very thick] (10,0.5) rectangle (10.5,1);
		\filldraw[blue!60!white, draw=black, very thick] (10.5,1.5) rectangle (12,2);
		\filldraw[blue!60!white, draw=black, very thick] (12,2) rectangle (12.45,2.5);
		\filldraw[yellow!60!white, draw=black, very thick] (12,4.5) rectangle (12.45,5);
		\filldraw[yellow!60!white, draw=black, very thick] (11.5,5) rectangle (12,5.5);
		\filldraw[yellow!60!white, draw=black, very thick] (11,5.5) rectangle (11.5,6);
		\filldraw[yellow!60!white, draw=black, very thick] (10.5,6) rectangle (11,6.45);
		\filldraw[purple!60!white, draw=black, very thick] (5.05,5) rectangle (5.5,5.5);
	\end{tikzpicture}
	\caption{An example of a representation of a tree}
	\label{fig:treeexample1}
\end{figure}
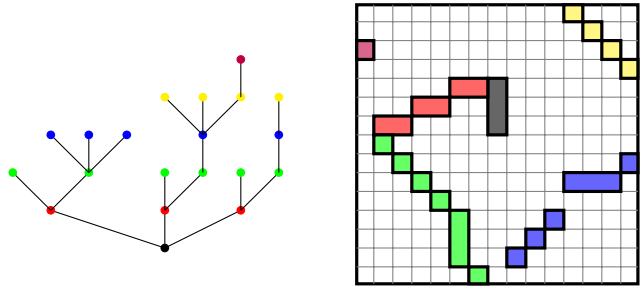

\section{Grid Graph}\label{secgrid}

It would probably be difficult to classify every planar graph, but we can find some natural planar graphs which are TRVGs.
In this section, we will consider a rectangular, a triangular and a hexagonal grid graph.

\begin{proof}[Proof of Theorem~\ref{thm:istrvg}$(iv)$]
For $i\in\Z$ and $j\in\Z$, let $v_{i,j}$ be the vertex in column $i$ and row $j$.
We will construct a rectangle $R_{i,j}$ for $v_{i,j}$ inside a bounding box $B_{i,j}$ which is a unit square centered at the point $(i,j)$ in the plane (see Figure~\ref{fig:rectangleside}).

We would like the center of $R_{i,j}$ to be lower than that of $R_{i-1,j}$ and on the right of that of $R_{i,j-1}$ for all $i,j\in\Z$.
First we construct $R_{0,0}$ in $B_{0,0}$ 
not touching the boundary.
Suppose we have constructed $R_{i,j}$ for $|i|+|j|\le k$.
Now we will construct $R_{i,j}$ for $|i|+|j|=k+1$.
Observe that each $v_{i,j}$ is adjacent to at most two vertices whose rectangles have previously been constructed.
Considering only the previously constructed rectangles, we can construct $R_{i,j}$ such that $(i)$ it sees precisely the rectangles for the neighbors of $v_{i,j}$ but not any others, $(ii)$ it does not touch the boundary of $B_{i,j}$ and $(iii)$ it does not lie entirely in the line of sight of the rectangle for any neighbor of $v_{i,j}$.
The conditions $(ii)$ and $(iii)$ ensure that the rectangles in the next step can be constructed to satisfy $(i)$.
\end{proof}
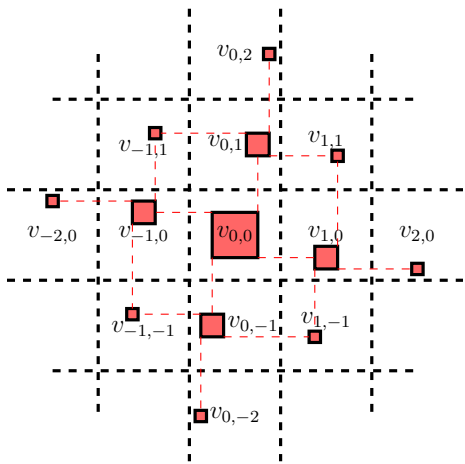
\begin{figure}[h]
	\centering
	\begin{tikzpicture} [scale=0.6]
		\filldraw[red!60!white, draw=black, very thick] (-0.5,-0.5) rectangle (0.5,0.5);
		\filldraw[red!60!white, draw=black, very thick] (0.25,1.75) rectangle (0.75,2.25);
		\filldraw[red!60!white, draw=black, very thick] (0.625,3.875) rectangle (0.875,4.125);
		\filldraw[red!60!white, draw=black, very thick] (-0.25,-1.75) rectangle (-0.75,-2.25);
		\filldraw[red!60!white, draw=black, very thick] (-0.625,-3.875) rectangle (-0.875,-4.125);
		\filldraw[red!60!white, draw=black, very thick] (1.75,-0.25) rectangle (2.25,-0.75);
		\filldraw[red!60!white, draw=black, very thick] (3.875,-0.625) rectangle (4.125,-0.875);
		\filldraw[red!60!white, draw=black, very thick] (-1.75,0.25) rectangle (-2.25,0.75);
		\filldraw[red!60!white, draw=black, very thick] (-3.875,0.625) rectangle (-4.125,0.875);
		
		\filldraw[red!60!white, draw=black, very thick] (2.125,1.875) rectangle (2.375,1.625);
		\filldraw[red!60!white, draw=black, very thick] (-2.125,-1.875) rectangle (-2.375,-1.625);
		\filldraw[red!60!white, draw=black, very thick] (1.625,-2.125) rectangle (1.875,-2.375);
		\filldraw[red!60!white, draw=black, very thick] (-1.875,2.125) rectangle (-1.625,2.375);
		
		\draw[dashed, very thick] (-1,5) -- (-1,-5);
		\draw[dashed, very thick] (1,5) -- (1,-5);
		\draw[dashed, very thick] (-5,-1) -- (5,-1);
		\draw[dashed, very thick] (-5,1) -- (5,1);
		\draw[dashed, very thick] (-4,-3) -- (4,-3);
		\draw[dashed, very thick] (-4,3) -- (4,3);
		\draw[dashed, very thick] (-3,4) -- (-3,-4);
		\draw[dashed, very thick] (3,4) -- (3,-4);
		
		\draw (0,0) node[rectangle,scale=0.8] {$v_{0,0}$};
		\draw (-0.2,2) node[rectangle,scale=0.8] {$v_{0,1}$};
		\draw (0.4,-2) node[rectangle,scale=0.8] {$v_{0,-1}$};
		\draw (2,0) node[rectangle,scale=0.8] {$v_{1,0}$};
		\draw (-2,0) node[rectangle,scale=0.8] {$v_{-1,0}$};
		\draw (-4,0) node[rectangle,scale=0.8] {$v_{-2,0}$};
		\draw (4,0) node[rectangle,scale=0.8] {$v_{2,0}$};
		\draw (0,4) node[rectangle,scale=0.8] {$v_{0,2}$};
		\draw (0,-4) node[rectangle,scale=0.8] {$v_{0,-2}$};
		\draw (-2,1.9) node[rectangle,scale=0.8] {$v_{-1,1}$};
		\draw (-2,-2.1) node[rectangle,scale=0.8] {$v_{-1,-1}$};
		\draw (2,2.1) node[rectangle,scale=0.8] {$v_{1,1}$};
		\draw (2,-1.9) node[rectangle,scale=0.8] {$v_{1,-1}$};
		
		\draw[dashed, red] (0.5,0.5) -- (0.5,1.75);
		\draw[dashed, red] (0.75,2.25) -- (0.75,3.875);
		\draw[dashed, red] (-0.5,-0.5) -- (-0.5,-1.75);
		\draw[dashed, red] (-0.75,-2.25) -- (-0.75,-3.875);
		\draw[dashed, red] (0.5,-0.5) -- (1.75,-0.5);
		\draw[dashed, red] (2.25,-0.75) -- (3.875,-0.75);
		\draw[dashed, red] (-0.5,0.5) -- (-1.75,0.5);
		\draw[dashed, red] (-2.25,0.75) -- (-3.875,0.75);
		\draw[dashed, red] (2.25,-0.25) -- (2.25,1.625);
		\draw[dashed, red] (1.75,-0.75) -- (1.75,-2.125);
		\draw[dashed, red] (-2.25,0.25) -- (-2.25,-1.625);
		\draw[dashed, red] (-1.75,0.75) -- (-1.75,2.125);
		\draw[dashed, red] (0.75,1.75) -- (2.125,1.75);
		\draw[dashed, red] (0.25,2.25) -- (-1.625,2.25);
		\draw[dashed, red] (-0.25,-2.25) -- (1.625,-2.25);
		\draw[dashed, red] (-0.75,-1.75) -- (-2.125,-1.75);
	\end{tikzpicture}
	\caption{An example of representation of an infinite square grid graph}
	\label{fig:rectangleside}
\end{figure}



\begin{figure}[h]
	\centering
	\begin{tikzpicture}
		\filldraw (0,0) circle(0.1cm);
		\filldraw (1,0) circle(0.1cm);
		\filldraw (-1,0) circle(0.1cm);
		\filldraw (-0.5,0.86) circle(0.1cm);
		\filldraw (-0.5,-0.86) circle(0.1cm);
		\filldraw (0.5,0.86) circle(0.1cm);
		\filldraw (0.5,-0.86) circle(0.1cm);
		\filldraw (2,0) circle(0.1cm);
		\filldraw (-2,0) circle(0.1cm);
		\filldraw (1,1.73) circle(0.1cm);
		\filldraw (-1,1.73) circle(0.1cm);
		\filldraw (1,-1.73) circle(0.1cm);
		\filldraw (-1,-1.73) circle(0.1cm);
		\filldraw (-1.5,0.86) circle(0.1cm);
		\filldraw (1.5,0.86) circle(0.1cm);
		\filldraw (-1.5,-0.86) circle(0.1cm);
		\filldraw (1.5,-0.86) circle(0.1cm);
		\filldraw (0,1.73) circle(0.1cm);
		\filldraw (0,-1.73) circle(0.1cm);
		
		\draw[thick] (-2.3,0) -- (2.3,0);
		\draw[thick] (-1.8,0.86) -- (1.8,0.86);
		\draw[thick] (-1.8,-0.86) -- (1.8,-0.86);
		\draw[thick] (-1.3,1.73) -- (1.3,1.73);
		\draw[thick] (-1.3,-1.73) -- (1.3,-1.73);
		\draw[thick] (-1.2,-2.07) -- (1.2,2.07);
		\draw[thick] (-1.7,-1.2) -- (0.2,2.07);
		\draw[thick] (1.7,1.2) -- (-0.2,-2.07);
		\draw[thick] (-2.2,-0.34) -- (-0.8,2.07);
		\draw[thick] (2.2,0.34) -- (0.8,-2.07);
		\draw[thick] (1.2,-2.07) -- (-1.2,2.07);
		\draw[thick] (-1.7,1.2) -- (0.2,-2.07);
		\draw[thick] (1.7,-1.2) -- (-0.2,2.07);
		\draw[thick] (-2.2,0.34) -- (-0.8,-2.07);
		\draw[thick] (2.2,-0.34) -- (0.8,2.07);
		
		\draw[dotted, thick] (2.3,0) -- (2.8,0);
		\draw[dotted, thick] (-2.3,0) -- (-2.8,0);
		\draw[dotted, thick] (-1.8,0.86) -- (-2.3,0.86);
		\draw[dotted, thick] (1.8,0.86) -- (2.3,0.86);
		\draw[dotted, thick] (-1.8,-0.86) -- (-2.3,-0.86);
		\draw[dotted, thick] (1.8,-0.86) -- (2.3,-0.86);
		\draw[dotted, thick] (1.3,1.73) -- (1.8,1.73);
		\draw[dotted, thick] (1.3,-1.73) -- (1.8,-1.73);
		\draw[dotted, thick] (-1.3,1.73) -- (-1.8,1.73);
		\draw[dotted, thick] (-1.3,-1.73) -- (-1.8,-1.73);
		
		\draw[dotted, thick] (-1.7,-1.2) -- (-1.9,-1.55);
		\draw[dotted, thick] (0.2,2.07) -- (0.4,2.42);
		\draw[dotted, thick] (-1.7,1.2) -- (-1.9,1.55);
		\draw[dotted, thick] (0.2,-2.07) -- (0.4,-2.42);
		\draw[dotted, thick] (1.7,-1.2) -- (1.9,-1.55);
		\draw[dotted, thick] (-0.2,2.07) -- (-0.4,2.42);
		\draw[dotted, thick] (1.7,1.2) -- (1.9,1.55);
		\draw[dotted, thick] (-0.2,-2.07) -- (-0.4,-2.42);
		\draw[dotted, thick] (-2.2,-0.34) -- (-2.4,-0.7);
		\draw[dotted, thick] (1.2,2.07) -- (1.4,2.42);
		\draw[dotted, thick] (2.2,-0.34) -- (2.4,-0.7);
		\draw[dotted, thick] (-1.2,2.07) -- (-1.4,2.42);
		\draw[dotted, thick] (-2.2,0.34) -- (-2.4,0.7);
		\draw[dotted, thick] (1.2,-2.07) -- (1.4,-2.42);
		\draw[dotted, thick] (2.2,0.34) -- (2.4,0.7);
		\draw[dotted, thick] (-1.2,-2.07) -- (-1.4,-2.42);
		\draw[dotted, thick] (0.8,2.07) -- (0.6,2.42);
		\draw[dotted, thick] (-0.8,-2.07) -- (-0.6,-2.42);
		\draw[dotted, thick] (0.8,-2.07) -- (0.6,-2.42);
		\draw[dotted, thick] (-0.8,2.07) -- (-0.6,2.42);
		
		\draw (0,0.25) node[rectangle] {$v_{0,0}$};
		\draw (-1,0.25) node[rectangle] {$v_{-1,0}$};
		\draw (1,0.25) node[rectangle] {$v_{1,0}$};
		\draw (-2,0.25) node[rectangle] {$v_{-2,0}$};
		\draw (2,0.25) node[rectangle] {$v_{2,0}$};
		\draw (-0.5,1.11) node[rectangle] {$v_{-1,1}$};
		\draw (0.5,1.11) node[rectangle] {$v_{0,1}$};
		\draw (-1.5,1.11) node[rectangle] {$v_{-2,1}$};
		\draw (1.5,1.11) node[rectangle] {$v_{1,1}$};
		\draw (-0.5,-0.61) node[rectangle] {$v_{0,-1}$};
		\draw (0.5,-0.61) node[rectangle] {$v_{1,-1}$};
		\draw (-1.5,-0.61) node[rectangle] {$v_{-1,-1}$};
		\draw (1.5,-0.61) node[rectangle] {$v_{2,-1}$};
		\draw (0,1.96) node[rectangle] {$v_{-1,2}$};
		\draw (1,1.96) node[rectangle] {$v_{0,2}$};
		\draw (-1,1.96) node[rectangle] {$v_{-2,2}$};
		\draw (0,-1.48) node[rectangle] {$v_{1,-2}$};
		\draw (1,-1.48) node[rectangle] {$v_{2,-2}$};
		\draw (-1,-1.48) node[rectangle] {$v_{0,-2}$};
		
		\draw (3.45,0) node[rectangle] {row $0$};
		\draw (3.45,0.86) node[rectangle] {row $1$};
		\draw (3.45,1.73) node[rectangle] {row $2$};
		\draw (3.62,-0.86) node[rectangle] {row $-1$};
		\draw (3.62,-1.73) node[rectangle] {row $-2$};
	\end{tikzpicture}
	\caption{An infinite triangular grid graph}
	\label{fig:infinitytriangulartile}
\end{figure}
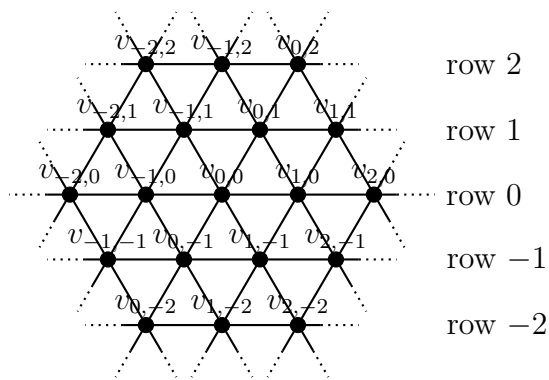
\begin{proof}[Proof of Theorem~\ref{thm:istrvg}$(v)$]
For $i\in\Z$ and $j\in\Z$, let $v_{i,j}$ be the vertex in column $i$ and row $j$ where we take one of the two tilted vertical directions to be a column (as shown in Figure~\ref{fig:infinitytriangulartile}).
We will construct a rectangle $R_{i,j}$ for $v_{i,j}$ inside a bounding box $B_{j}$ for row $j$ which is a unit square centered at the point $(\lceil\frac{j}{2}\rceil,\lfloor\frac{j}{2}\rfloor)$ in the plane (see Figure~\ref{fig:trueboxtri}).
Note that rectangles in $B_j$ can only see rectangles in $B_{j-1}$ and $B_{j+1}$.

For each even row $j$, we construct rectangles in $B_{j}$ representing an infinite path such that two adjacent rectangles see each other horizontally (as shown in Figure~\ref{fig:infinitytriangulartile}).

For each odd row $j$, since $v_{i,j}$ is adjacent to $v_{i,j-1},v_{i+1,j-1}$ in row $j-1$ and $v_{i-1,j+1},v_{i,j+1}$ in row $j+1$, rectangle $R_{i,j}$ must see $R_{i,j-1},R_{i+1,j-1}$ in $B_{j-1}$ and $R_{i-1,j+1},R_{i,j+1}$ in $B_{j+1}$.
So we take $R_{i,j}$ to be the intersection of the intersection of the horizontal lines of sight of $R_{i,j-1}$ and $R_{i+1,j-1}$, and the union of the vertical lines of sight of $R_{i-1,j+1}$ and $R_{i,j+1}$.
It is easy to check that the rectangles in $B_{j}$ form an infinite path.
\begin{figure}[h]
	\centering
	\begin{tikzpicture} [scale=0.6]
		\filldraw[blue!60!white, draw=black, very thick] (0,0) rectangle (2,2);
		\filldraw[blue!60!white, draw=black, very thick] (2,0) rectangle (4,2);
		\filldraw[blue!60!white, draw=black, very thick] (2,2) rectangle (4,4);
		\filldraw[blue!60!white, draw=black, very thick] (4,2) rectangle (6,4);
		\filldraw[blue!60!white, draw=black, very thick] (4,4) rectangle (6,6);
		\filldraw[blue!60!white, draw=black, very thick] (6,4) rectangle (8,6);
		
		\draw (1,1.5) node[rectangle, scale=0.6] {Bounding};
		\draw (1,1) node[rectangle, scale=0.6] {box of};
		\draw (1,0.5) node[rectangle, scale=0.6] {row $-2$};
		\draw (3,1.5) node[rectangle, scale=0.6] {Bounding};
		\draw (3,1) node[rectangle, scale=0.6] {box of};
		\draw (3,0.5) node[rectangle, scale=0.6] {row $-1$};
		\draw (3,3.5) node[rectangle, scale=0.6] {Bounding};
		\draw (3,3) node[rectangle, scale=0.6] {box of};
		\draw (3,2.5) node[rectangle, scale=0.6] {row $0$};
		\draw (5,3.5) node[rectangle, scale=0.6] {Bounding};
		\draw (5,3) node[rectangle, scale=0.6] {box of};
		\draw (5,2.5) node[rectangle, scale=0.6] {row $1$};
		\draw (5,5.5) node[rectangle, scale=0.6] {Bounding};
		\draw (5,5) node[rectangle, scale=0.6] {box of};
		\draw (5,4.5) node[rectangle, scale=0.6] {row $2$};
		\draw (7,5.5) node[rectangle, scale=0.6] {Bounding};
		\draw (7,5) node[rectangle, scale=0.6] {box of};
		\draw (7,4.5) node[rectangle, scale=0.6] {row $3$};
		
		\draw[loosely dotted, scale=2, ultra thick] (-0.05,-0.05) -- (-0.3,-0.3);
		\draw[loosely dotted, scale=2, ultra thick] (4.05,3.05) -- (4.3,3.3);
	\end{tikzpicture}
	\caption{A placement of the bounding boxes of rows}
	\label{fig:trueboxtri}
\end{figure}
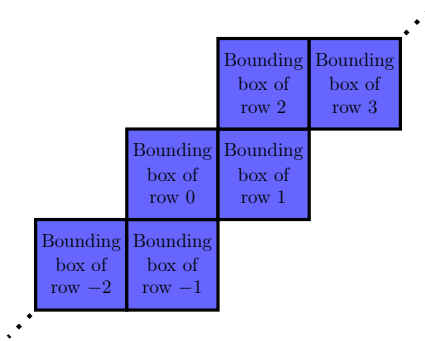
\begin{figure}[h]
	\centering
	\begin{tikzpicture} [scale=1]
		\filldraw[red!60!white, draw=black, very thick] (0,0) rectangle (1,1);
		\filldraw[red!60!white, draw=black, very thick] (1,0.75) rectangle (1.5,1.25);
		\filldraw[red!60!white, draw=black, very thick] (1.5,1.125) rectangle (1.75,1.375);
		\filldraw[red!60!white, draw=black, very thick] (0,0.25) rectangle (-0.5,-0.25);
		\filldraw[red!60!white, draw=black, very thick] (-0.5,-0.125) rectangle (-0.75,-0.375);
		\filldraw[red!60!white, draw=black, very thick] (-0.875,-0.4375) rectangle (-0.75,-0.3125);
		\filldraw[red!60!white, draw=black, very thick] (1.75,1.3125) rectangle (1.875,1.4375);
		\draw[dashed, red, very thick] (-1,-1) rectangle (2,2);
		
		\filldraw[red!60!white, draw=black, very thick] (3,3) rectangle (4,4);
		\filldraw[red!60!white, draw=black, very thick] (4,3.75) rectangle (4.5,4.25);
		\filldraw[red!60!white, draw=black, very thick] (4.5,4.125) rectangle (4.75,4.375);
		\filldraw[red!60!white, draw=black, very thick] (3,3.25) rectangle (2.5,2.75);
		\filldraw[red!60!white, draw=black, very thick] (2.5,2.875) rectangle (2.25,2.625);
		\filldraw[red!60!white, draw=black, very thick] (2.125,2.5625) rectangle (2.25,2.6875);
		\filldraw[red!60!white, draw=black, very thick] (4.75,4.3125) rectangle (4.875,4.4375);
		\draw[dashed, red, very thick] (2,2) rectangle (5,5);
		
		\filldraw[red!60!white, draw=black, very thick] (6,6) rectangle (7,7);
		\filldraw[red!60!white, draw=black, very thick] (7,6.75) rectangle (7.5,7.25);
		\filldraw[red!60!white, draw=black, very thick] (7.5,7.125) rectangle (7.75,7.375);
		\filldraw[red!60!white, draw=black, very thick] (6,6.25) rectangle (5.5,5.75);
		\filldraw[red!60!white, draw=black, very thick] (5.5,5.875) rectangle (5.25,5.625);
		\filldraw[red!60!white, draw=black, very thick] (5.125,5.5625) rectangle (5.25,5.6875);
		\filldraw[red!60!white, draw=black, very thick] (7.75,7.3125) rectangle (7.875,7.4375);
		\draw[dashed, red, very thick] (5,5) rectangle (8,8);
		
		\filldraw[red!60!white, draw=black, thick] (5.125,2.6875) rectangle (5.5,2.625);
		\filldraw[red!60!white, draw=black, very thick] (5.25,2.875) rectangle (6,2.75);
		\filldraw[red!60!white, draw=black, very thick] (5.5,3) rectangle (7,3.25);
		\filldraw[red!60!white, draw=black, very thick] (6,3.75) rectangle (7.5,4);
		\filldraw[red!60!white, draw=black, very thick] (7,4.25) rectangle (7.75,4.125);
		\filldraw[red!60!white, draw=black, thick] (7.5,4.375) rectangle (7.875,4.3125);
		
		\filldraw[red!60!white, draw=black, thick] (2.125,-0.3125) rectangle (2.5,-0.375);
		\filldraw[red!60!white, draw=black, very thick] (2.25,-0.125) rectangle (3,-0.25);
		\filldraw[red!60!white, draw=black, very thick] (2.5,0) rectangle (4,0.25);
		\filldraw[red!60!white, draw=black, very thick] (3,0.75) rectangle (4.5,1);
		\filldraw[red!60!white, draw=black, very thick] (4,1.25) rectangle (4.75,1.125);
		\filldraw[red!60!white, draw=black, thick] (4.5,1.375) rectangle (4.875,1.3125);
		
		\filldraw[red!60!white, draw=black, thick] (8.125,5.6875) rectangle (8.5,5.625);
		\filldraw[red!60!white, draw=black, very thick] (8.25,5.875) rectangle (9,5.75);
		\filldraw[red!60!white, draw=black, very thick] (8.5,6) rectangle (10,6.25);
		\filldraw[red!60!white, draw=black, very thick] (9,6.75) rectangle (10.5,7);
		\filldraw[red!60!white, draw=black, very thick] (10,7.25) rectangle (10.75,7.125);
		\filldraw[red!60!white, draw=black, thick] (10.5,7.375) rectangle (10.875,7.3125);
		
		\draw[dashed, blue, very thick] (2,-1) rectangle (5,2);
		\draw[dashed, blue, very thick] (5,2) rectangle (8,5);
		\draw[dashed, blue, very thick] (8,5) rectangle (11,8);
		
		\draw[dashed] (5.125,8) -- (5.125,2);
		\draw[dashed] (5.25,8) -- (5.25,2);
		\draw[dashed] (5.5,8) -- (5.5,2);
		\draw[dashed] (6,8) -- (6,2);
		\draw[dashed] (7,8) -- (7,2);
		\draw[dashed] (7.5,8) -- (7.5,2);
		\draw[dashed] (7.75,8) -- (7.75,2);
		\draw[dashed] (7.875,8) -- (7.875,2);
		
		\draw[dashed] (2,4.375) -- (8,4.375);
		\draw[dashed] (2,4.3125) -- (8,4.3125);
		\draw[dashed] (2,4.25) -- (8,4.25);
		\draw[dashed] (2,4.125) -- (8,4.125);
		\draw[dashed] (2,4) -- (8,4);
		\draw[dashed] (2,3.75) -- (8,3.75);
		\draw[dashed] (2,3.25) -- (8,3.25);
		\draw[dashed] (2,3) -- (8,3);
		\draw[dashed] (2,2.75) -- (8,2.75);
		\draw[dashed] (2,2.875) -- (8,2.875);
		\draw[dashed] (2,2.625) -- (8,2.625);
		\draw[dashed] (2,2.6875) -- (8,2.6875);
		
		\draw[loosely dotted, ultra thick] (-1.1,-1.1) -- (-1.4,-1.4);
		\draw[loosely dotted, ultra thick] (11.1,8.1) -- (11.4,8.4);
	\end{tikzpicture}
	\caption{A representation of an infinite triangular grid graph}
	\label{fig:triangulartrvg}
\end{figure}
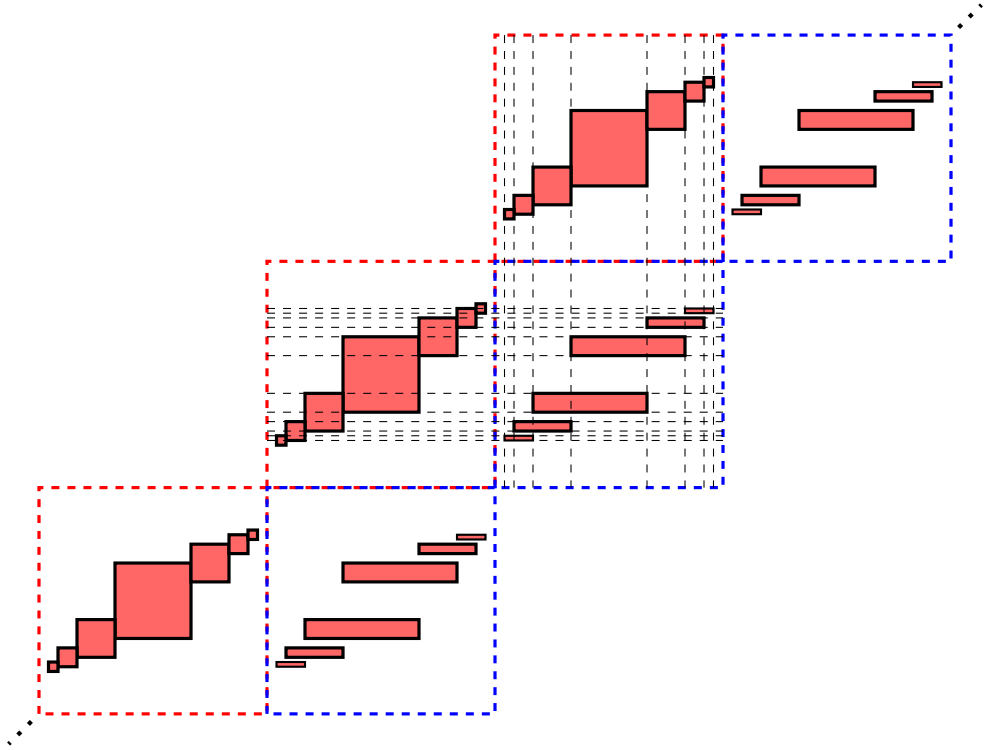
\end{proof}

\begin{proof}[Proof of Theorem~\ref{thm:istrvg}$(vi)$]
This follows from Theorem~\ref{thm:istrvg}$(v)$ since the infinite hexagonal grid graph is an induced subgraph of an infinite triangular grid graph obtained by removing the vertices $v_{a+3b,a}$ for all $a,b \in \Z$.
\end{proof}

\section{Bipartite Graph}\label{secbipartite}

For any triangle-free TRVG, we can easily see that it is also an RVG.
Now, let us consider whether the converse holds for a triangle-free graph.
We restrict our attention to bipartite graphs.
In this section, we will prove Theorems~\ref{thm:bipartTRVG},~\ref{thm:compbip},~\ref{thm:bipartTRVGtorus} and~\ref{thm:compbiptorus}.

We color the rectangles for the vertices of a bipartite graph with green or red depending on which part they are in.
Denote the set of green rectangles by $\G$ and that of red rectangles by $\R$.

To prove Theorem~\ref{thm:bipartTRVG}, the following lemma will help us determine the maximum number of edges of a general bipartite TRVG with $n$ vertices.

\begin{lemma}\label{lemma:gseer}
Let $G$ be a bipartite TRVG represented by $p$ green rectangles $g_1, g_2, \dots, g_p$ in $\G$ and $q$ red rectangles $r_1, r_2, \dots, r_q$ in $\R$.
If $g_i$ sees $\alpha_i$ red rectangles horizontally for $i \in \{1, 2, \dots, p\}$, then
\[q \ge \alpha_1+\alpha_2+\dots+\alpha_p-(p-1).\]
\end{lemma}

\begin{proof}
Since any two rectangles in $\G$ do not see each other, we can relabel the green rectangles in $\G$ such that $g_i$ is higher than $g_{i+1}$ for all $i \in \{1, 2, \dots, p-1\}$.

First $g_1$ sees $\alpha_1$ red rectangles horizontally.
Observe that $g_1$ and $g_2$ see horizontally at most one red rectangle in common otherwise the red rectangles would see each other (see Figure~\ref{fig:gseerlem}).
So the number of red rectangles seen horizontally by $g_1$ or $g_2$ is at least $\alpha_1 + \alpha_2 -1$.
Since $g_3$ sees at most one red rectangle seen by $g_1$ or $g_2$, the number of red rectangles seen horizontally by $g_1,g_2$ or $g_3$ is at least $(\alpha_1 + \alpha_2 -1)+ \alpha_3 -1 = \alpha_1 + \alpha_2 + \alpha_3 -2$.
Similarly, the number of red rectangles seen horizontally by $g_1,g_2,\dots,g_{p-1}$ or $g_p$ is at least $\alpha_1 + \alpha_2 + \dots + \alpha_p -(p-1)$.
Hence, the result follows.
\end{proof}
\begin{figure}[h]
	\centering
	\begin{tikzpicture}[scale=0.65]
		\filldraw[green!60!white, draw=black, very thick] (0,0) rectangle (1,1.5);
		\filldraw[green!60!white, draw=black, very thick] (2,-2) rectangle (3.2,-0.2);
		\filldraw[red!60!white, draw=black, very thick] (1.5,1.1) rectangle (2,1.75);
		\filldraw[red!60!white, draw=black, very thick] (2.25,0.2) rectangle (2.75,0.8);
		\filldraw[red!60!white, draw=black, very thick] (-0.5,-0.75) rectangle (0,-0.25);
		\filldraw[red!60!white, draw=black, very thick] (0,-1.25) rectangle (0.6,-0.75);
		\filldraw[red!60!white, draw=black, very thick] (3.5,-2) rectangle (4,-1.5);
		\draw (0.5,0.75) node[rectangle] {$g_1$};
		\draw (2.6,-1.1) node[rectangle] {$g_2$};
		\draw[dashed] (1,1.3) -- (1.5,1.3);
		\draw[dashed] (1,0.5) -- (2.25,0.5);
		\draw[dashed] (0,-0.5) -- (2,-0.5);
		\draw[dashed] (0.6,-1) -- (2,-1);
		\draw[dashed] (3.2,-1.75) -- (3.5,-1.75);
		
		\draw[very thick] (5,1.95) -- (5,-2.2);
		
		\filldraw[green!60!white, draw=black, very thick] (7,0) rectangle (8,1.5);
		\filldraw[green!60!white, draw=black, very thick] (9,-2) rectangle (10.2,-0.2);
		\filldraw[red!60!white, draw=black, very thick] (8.5,1.1) rectangle (9,1.75);
		\filldraw[red!60!white, draw=black, very thick] (9.25,0.4) rectangle (9.75,1);
		\filldraw[red!60!white, draw=black, very thick] (6.2,-0.7) rectangle (6.7,0.3);
		\filldraw[red!60!white, draw=black, very thick] (7,-1.25) rectangle (7.75,-0.75);
		\filldraw[red!60!white, draw=black, very thick] (10.5,-2) rectangle (11,-1.5);
		\draw (7.5,0.75) node[rectangle] {$g_1$};
		\draw (9.6,-1.1) node[rectangle] {$g_2$};
		\draw[dashed] (8,1.3) -- (8.5,1.3);
		\draw[dashed] (8,0.7) -- (9.25,0.7);
		\draw[dashed] (6.7,0.15) -- (7,0.15);
		\draw[dashed] (6.7,-0.45) -- (9,-0.45);
		\draw[dashed] (7.75,-1) -- (9,-1);
		\draw[dashed] (10.2,-1.75) -- (10.5,-1.75);
	\end{tikzpicture}
	\caption{An example of Lemma~\ref{lemma:gseer}}
	\label{fig:gseerlem}
\end{figure}
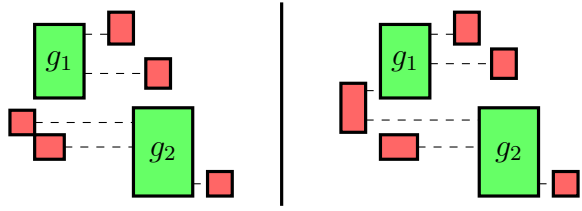	

\begin{proof}[Proof of Theorem~\ref{thm:bipartTRVG}]
Let $G$ be a bipartite graph on $n$ vertices and let $g_1, g_2, \dots, g_p$ be green rectangles representing the first part.
Suppose $g_i$ sees $\beta_i$ red rectangles vertically and $g_i$ sees $\gamma_i$ red rectangles horizontally for $i \in \{1, 2, \dots, p\}$.
By Lemma~\ref{lemma:gseer}, we have
\[n-p \ge \beta_1+\beta_2+\dots+\beta_p-(p-1)\]
and
\[n-p \ge \gamma_1+\gamma_2+\dots+\gamma_p-(p-1).\]
So
\begin{align*}
2(n-p) & \ge(\beta_1+\beta_2+\dots+\beta_p+\gamma_1+\gamma_2+\dots+\gamma_p)-2(p-1)\\
    & \ge e(G)-2p+2
\end{align*}
where $e(G)$ is the number of edges of $G$.
Hence, $2n-2 \ge e(G)$.
Thus, a bipartite TRVG with $n$ vertices has at most $2n-2$ edges.

This is best possible for $n\ge7$ since we can construct a bipartite TRVG with $n\ge7$ vertices that has exactly $2n-2$ edges (see Figure~\ref{fig:n-4and4}).
\begin{figure}[h]
	\centering
	\begin{tikzpicture}[scale=0.6]
		\draw[very thick] (0.1,0.6) rectangle (4.9,5.4);
		\draw[step=0.5cm,gray,dashed] (0.1,0.6) grid (4.9,5.4);
		\filldraw[red!60!white, draw=black, very thick] (0.5,3) rectangle (1.5,4); 
		\filldraw[red!60!white, draw=black, very thick] (1.5,4) rectangle (2.5,5);
		\filldraw[red!60!white, draw=black, very thick] (2.5,1) rectangle (3.5,2);
		\filldraw[red!60!white, draw=black, very thick] (3.5,2) rectangle (4.5,3);
		\filldraw[green!60!white, draw=black, very thick] (2,2.5) rectangle (3,3.5);
		\filldraw[green!60!white, draw=black, very thick] (3,3.5) rectangle (4,4.5);
		\filldraw[green!60!white, draw=black,very thick] (1.3,1.8) rectangle (2,2.5);
		\filldraw[green!60!white, draw=black,very thick] (0.5,1) rectangle (0.75,1.25);
		\filldraw[green!60!white, draw=black,very thick] (1.3,1.8) rectangle (1.05,1.55);
		\draw[dotted, very thick] (1,1.5) -- (0.8,1.3);
	\end{tikzpicture}
	\caption{A representation of a bipartite graph with $2n-2$ edges}
	\label{fig:n-4and4}
\end{figure}
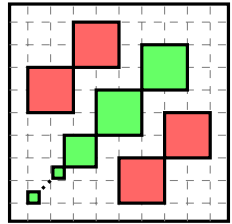
\end{proof}

However, the converse of Theorem~\ref{thm:bipartTRVG} is not true by considering a large bipartite graph with fewer than $2n-2$ edges containing $K_{4,4}$ as an induced subgraph.

\begin{proof}[Proof of Theorem~\ref{thm:compbip}]
    Figure~\ref{fig:compbip} shows representations of $K_{1,q},K_{2,q}$ and $K_{3,4}$. Since $K_{3,3}$ is an induced subgraph of $K_{3,4}$, it is a TRVG.

    Conversely, $K_{3,5}$ and $K_{4,4}$ are non-TRVG by Theorem~\ref{thm:bipartTRVG} since $2|K_{3,5}|-2=14<15=e(K_{3,5})$ and $2|K_{4,4}|-2=14<16=e(K_{4,4})$.
    For $p\le q,$ if $p\ge 3$ and $(p,q)\not\in\{(3,3),(3,4)\},$ then $K_{p,q}$ contains either $K_{3,5}$ or $K_{4,4}$ as an induced subgraph, and hence it is a non-TRVG.
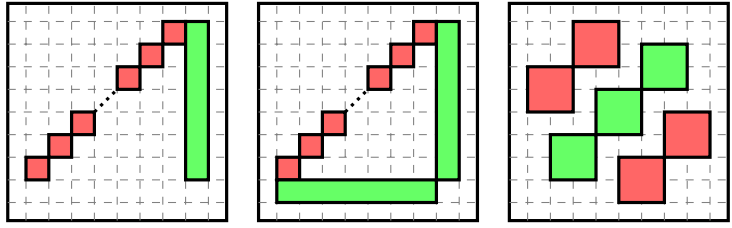
\begin{figure}[h]
	\centering
	\begin{tikzpicture}[scale=0.6]
	\draw[very thick] (-2.4,-3.4) rectangle (2.4,1.4);
	\draw[step=0.5cm,gray,dashed] (-2.4,-3.4) grid (2.4,1.4);
	\filldraw[red!60!white, draw=black, very thick] (-2,-2.5) rectangle (-1.5,-2); 
	\filldraw[red!60!white, draw=black, very thick] (-1.5,-2) rectangle (-1,-1.5);
        \filldraw[red!60!white, draw=black, very thick] (-1,-1.5) rectangle (-0.5,-1);
        \draw[dotted,very thick] (-0.5,-1)--(0,-0.5);
	\filldraw[red!60!white, draw=black, very thick] (0,-0.5) rectangle (0.5,0);
	\filldraw[red!60!white, draw=black, very thick] (0.5,0) rectangle (1,0.5);
	\filldraw[red!60!white, draw=black, very thick] (1,0.5) rectangle (1.5,1);
	\filldraw[green!60!white, draw=black,very thick] (1.5,-2.5) rectangle (2,1);

        \draw[very thick] (3.1,-3.4) rectangle (7.9,1.4);
	\draw[step=0.5cm,gray,dashed] (3.1,-3.4) grid (7.9,1.4);
	\filldraw[red!60!white, draw=black, very thick] (3.5,-2.5) rectangle (4,-2); 
	\filldraw[red!60!white, draw=black, very thick] (4,-2) rectangle (4.5,-1.5);
        \filldraw[red!60!white, draw=black, very thick] (4.5,-1.5) rectangle (5,-1);
	\draw[dotted,very thick] (5,-1)--(5.5,-0.5);
	\filldraw[red!60!white, draw=black, very thick] (5.5,-0.5) rectangle (6,0);
	\filldraw[red!60!white, draw=black, very thick] (6,0) rectangle (6.5,0.5);
        \filldraw[red!60!white, draw=black, very thick] (6.5,0.5) rectangle (7,1);
	\filldraw[green!60!white, draw=black,very thick] (7,-2.5) rectangle (7.5,1);
	\filldraw[green!60!white, draw=black,very thick] (3.5,-3) rectangle (7,-2.5);

        \draw[very thick] (8.6,-3.4) rectangle (13.4,1.4);
	\draw[step=0.5cm,gray,dashed] (8.6,-3.4) grid (13.4,1.4);
	\filldraw[red!60!white, draw=black, very thick] (9,-1) rectangle (10,0); 
	\filldraw[red!60!white, draw=black, very thick] (10,0) rectangle (11,1);
	\filldraw[red!60!white, draw=black, very thick] (11,-3) rectangle (12,-2);
	\filldraw[red!60!white, draw=black, very thick] (12,-2) rectangle (13,-1);
	\filldraw[green!60!white, draw=black, very thick] (10.5,-1.5) rectangle (11.5,-0.5);
	\filldraw[green!60!white, draw=black, very thick] (11.5,-0.5) rectangle (12.5,0.5);
	\filldraw[green!60!white, draw=black,very thick] (9.5,-2.5) rectangle (10.5,-1.5);
	\end{tikzpicture}
	\caption{Representations of $K_{1,q},K_{2,q}$ and $K_{3,4}$}
	\label{fig:compbip}
\end{figure}
\end{proof}

The proof of Lemma~\ref{lemma:gseer} can be adapted to prove an analogous result for the torus.
The difference is $g_p$ and $g_1$ might see horizontally one red rectangle in common (see Figure~\ref{fig:gseerlemtorus}).

\begin{figure}[h]
	\centering
	\begin{tikzpicture}[scale=0.6]
    	\filldraw[green!60!white, draw=black, very thick] (0,1) rectangle (1,2.5);
		\filldraw[green!60!white, draw=black, very thick] (2,-2) rectangle (3.2,-0.2);
		\filldraw[red!60!white, draw=black, very thick] (1.3,2.1) rectangle (1.8,3.25);
            \filldraw[red!60!white, draw=black, very thick] (1.3,-1.75) rectangle (1.8,-2.5);
		\filldraw[red!60!white, draw=black, very thick] (2.25,1.2) rectangle (2.75,1.8);
		\filldraw[red!60!white, draw=black, very thick] (-0.5,-0.6) rectangle (0,-0.1);
		\filldraw[red!60!white, draw=black, very thick] (0,-1.15) rectangle (0.6,-0.6);
		\filldraw[red!60!white, draw=black, very thick] (3.5,-1.7) rectangle (4,-1.2);
		\draw (0.5,1.75) node[rectangle] {$g_1$};
		\draw (2.6,-1.1) node[rectangle] {$g_p$};
        \draw (1.55,2.75) node[rectangle,scale=0.8] {$r_1$};
        \draw (1.55,-2.2) node[rectangle,scale=0.8] {$r_1$};
		\draw[dashed] (1,2.3) -- (1.3,2.3);
		\draw[dashed] (1,1.5) -- (2.25,1.5);
		\draw[dashed] (0,-0.35) -- (2,-0.35);
		\draw[dashed] (0.6,-0.9) -- (2,-0.9);
		\draw[dashed] (3.2,-1.45) -- (3.5,-1.45);
            \draw[dashed] (1.8,-1.875) -- (2,-1.875);
            \draw (1.25,0.6) node{$\vdots$};
            \draw[very thick,->,gray!80!black] (-1,3.25)--(1.75,3.25);
            \draw[very thick,->,gray!80!black] (-1,-2.5)--(1.75,-2.5);
            \draw[very thick,->,gray!80!black] (-1,-2.5)--(-1,0.375);
            \draw[very thick,->,gray!80!black] (4.5,-2.5)--(4.5,0.375);
            \draw[very thick,gray!80!black] (-1,-2.5) rectangle (4.5,3.25);
	\end{tikzpicture}
	\caption{An example of Lemma~\ref{lemma:gseertorus}}
	\label{fig:gseerlemtorus}
\end{figure}
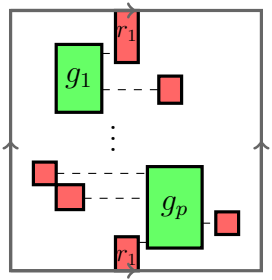	

\begin{lemma}\label{lemma:gseertorus}
Let $G$ be a bipartite TRVG with respect to the torus represented by $p$ green rectangles $g_1, g_2, \dots, g_p$ in $\G$ and $q$ red rectangles $r_1, r_2, \dots, r_q$ in $\R$.
If $g_i$ sees $\alpha_i$ red rectangles horizontally for $i \in \{1, 2, \dots, p\}$, then
\[q \ge \alpha_1+\alpha_2+\dots+\alpha_p-p.\]
\end{lemma}

We can use Lemma~\ref{lemma:gseertorus} to prove Theorem~\ref{thm:bipartTRVGtorus} similar to when we used Lemma~\ref{lemma:gseer} to prove Theorem~\ref{thm:bipartTRVG}.

\begin{proof}[Proof of Theorem~\ref{thm:bipartTRVGtorus}]
Let us use the same notations as the proof of Theorem~\ref{thm:bipartTRVG}.
By Lemma~\ref{lemma:gseertorus}, we have $n-p \ge \beta_1+\beta_2+\dots+\beta_p-p$ and $n-p \ge \gamma_1+\gamma_2+\dots+\gamma_p-p$.
So $2(n-p)\ge(\beta_1+\beta_2+\dots+\beta_p+\gamma_1+\gamma_2+\dots+\gamma_p)-2p\ge e(G)-2p$.
Hence, $2n \ge e(G)$.
Thus, a bipartite TRVG with respect to the torus on $n$ vertices has at most $2n$ edges.

This is best possible for $n\ge8$ since we can construct a bipartite TRVG with respect to the torus on $n$ vertices that has exactly $2n$ edges (see Figure~\ref{fig:n-4and4torus}).
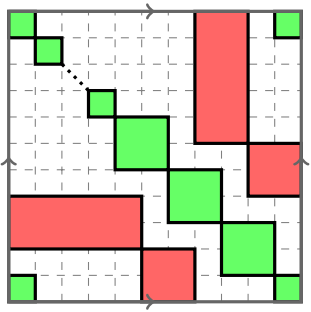
\begin{figure}[h]
	\centering
	\begin{tikzpicture} [scale=0.35]
        \draw[step=1cm,gray,dashed] (0,0) grid (11,11);
        \filldraw[red!60!white, draw=black, very thick] (0,2) rectangle (5,4); 
	\filldraw[red!60!white, draw=black, very thick] (5,0) rectangle (7,2);
	\filldraw[red!60!white, draw=black, very thick] (9,11) rectangle (7,6);
	\filldraw[red!60!white, draw=black, very thick] (9,6) rectangle (11,4);
	\filldraw[green!60!white, draw=black,very thick] (0,11) rectangle (1,10);
        \filldraw[green!60!white, draw=black,very thick] (1,10) rectangle (2,9);
        \draw[dotted,very thick] (2,9)--(3,8);
        \filldraw[green!60!white, draw=black,very thick] (3,8) rectangle (4,7);
        \filldraw[green!60!white, draw=black,very thick] (4,7) rectangle (6,5);
        \filldraw[green!60!white, draw=black,very thick] (6,5) rectangle (8,3);
        \filldraw[green!60!white, draw=black,very thick] (8,3) rectangle (10,1);
        \filldraw[green!60!white, draw=black,very thick] (10,1) rectangle (11,0);
        \filldraw[green!60!white, draw=black,very thick] (0,0) rectangle (1,1);
        \filldraw[green!60!white, draw=black,very thick] (10,10) rectangle (11,11);
        
        \draw[very thick,->,gray!80!black] (0,0)--(0,5.5);
        \draw[very thick,->,gray!80!black] (0,11)--(5.5,11);
        \draw[very thick,->,gray!80!black] (0,0)--(5.5,0);
        \draw[very thick,->,gray!80!black] (11,0)--(11,5.5);
        \draw[very thick,gray!80!black] (0,0) rectangle (11,11);
    \end{tikzpicture}
	\caption{A representation of a bipartite graph on $2n$ edges with respect to torus}
	\label{fig:n-4and4torus}
\end{figure}
\end{proof}

\begin{proof}[Proof of Theorem~\ref{thm:compbiptorus}]
Since a TRVG with respect to the plane is a TRVG with respect to the torus, by Theorem~\ref{thm:compbip}, $K_{p,q}$ ($p\le q$) is a TRVG with respect to the torus if $p\le 2$ or $(p,q)\in\{(3,3),(3,4)\}$.
Figure~\ref{fig:k36} shows representations of $K_{3,6}$ and $K_{4,4}$ with respect to the torus. Since $K_{3,5}$ is an induced subgraph of $K_{3,6}$, it is also a TRVG.

Conversely, $K_{3,7}$ and $K_{4,5}$ are non-TRVG with respect to the torus by Theorem~\ref{thm:bipartTRVGtorus} since $2|K_{3,7}|=20<21=e(K_{3,7})$ and $2|K_{4,5}|=18<20=e(K_{4,5})$.
For $p\le q,$ if $p\ge 3$ and $(p,q)\not\in\{(3,3),(3,4),(3,5),(3,6),(4,4)\},$ then $K_{p,q}$ contains either $K_{3,7}$ or $K_{4,5}$ as an induced subgraph, and hence it is a non-TRVG with respect to the torus.
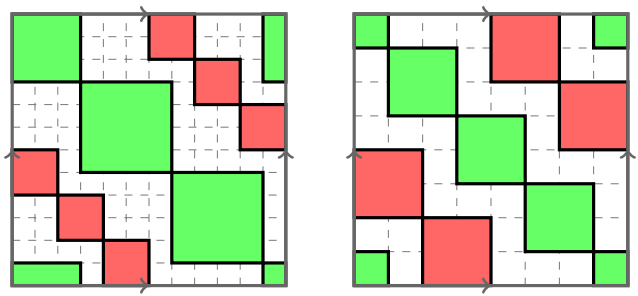
\begin{figure}[h]
	\centering
	\begin{tikzpicture} [scale=0.6]
	\draw[step=0.5cm,gray,dashed] (-3,-3) grid (3,3);
        \filldraw[red!60!white, draw=black, very thick] (0,3) rectangle (1,2); 
	\filldraw[red!60!white, draw=black, very thick] (1,2) rectangle (2,1);
        \filldraw[red!60!white, draw=black, very thick] (2,1) rectangle (3,0);
	\filldraw[red!60!white, draw=black, very thick] (-3,0) rectangle (-2,-1);
	\filldraw[red!60!white, draw=black, very thick] (-2,-1) rectangle (-1,-2);
	\filldraw[red!60!white, draw=black, very thick] (-1,-2) rectangle (0,-3);
	\filldraw[green!60!white, draw=black,very thick] (-3,1.5) rectangle (-1.5,3);
        \filldraw[green!60!white, draw=black,very thick] (-1.5,-0.5) rectangle (0.5,1.5);
        \filldraw[green!60!white, draw=black,very thick] (0.5,-2.5) rectangle (2.5,-0.5);
        \filldraw[green!60!white, draw=black,very thick] (2.5,-3) rectangle (3,-2.5);
        \filldraw[green!60!white, draw=black,very thick] (2.5,3) rectangle (3,1.5);
        \filldraw[green!60!white, draw=black,very thick] (-3,-3) rectangle (-1.5,-2.5);
        
        \draw[very thick,->,gray!80!black] (-3,-3)--(-3,0);
        \draw[very thick,->,gray!80!black] (-3,3)--(0,3);
        \draw[very thick,->,gray!80!black] (3,-3)--(3,0);
        \draw[very thick,->,gray!80!black] (-3,-3)--(0,-3);
        \draw[very thick,gray!80!black] (-3,-3) rectangle (3,3);
	
        \draw[step=0.75cm,gray,dashed] (4.5,-3) grid (10.5,3);
        \filldraw[red!60!white, draw=black, very thick] (7.5,3) rectangle (9,1.5); 
	\filldraw[red!60!white, draw=black, very thick] (9,1.5) rectangle (10.5,0);
	\filldraw[red!60!white, draw=black, very thick] (4.5,0) rectangle (6,-1.5);
	\filldraw[red!60!white, draw=black, very thick] (6,-1.5) rectangle (7.5,-3);
	\filldraw[green!60!white, draw=black,very thick] (4.5,2.25) rectangle (5.25,3);
        \filldraw[green!60!white, draw=black,very thick] (5.25,2.25) rectangle (6.75,0.75);
        \filldraw[green!60!white, draw=black,very thick] (6.75,0.75) rectangle (8.25,-0.75);
        \filldraw[green!60!white, draw=black,very thick] (8.25,-0.75) rectangle (9.75,-2.25);
        \filldraw[green!60!white, draw=black,very thick] (10.5,-3) rectangle (9.75,-2.25);
        \filldraw[green!60!white, draw=black,very thick] (10.5,3) rectangle (9.75,2.25);
        \filldraw[green!60!white, draw=black,very thick] (4.5,-3) rectangle (5.25,-2.25);
        
        \draw[very thick,->,gray!80!black] (4.5,-3)--(4.5,0);
        \draw[very thick,->,gray!80!black] (4.5,3)--(7.5,3);
        \draw[very thick,->,gray!80!black] (10.5,-3)--(10.5,0);
        \draw[very thick,->,gray!80!black] (4.5,-3)--(7.5,-3);
        \draw[very thick,gray!80!black] (4.5,-3) rectangle (10.5,3);
    \end{tikzpicture}
	\caption{Representations of $K_{3,6}$ and $K_{4,4}$ with respect to torus}
	\label{fig:k36}
\end{figure}
\end{proof}

\section{Power of Cycle}\label{secpoc}

In this section, we will observe that a cycle is a TRVG and prove Theorem~\ref{thm:poc}.
\begin{proof}[Proof of Theorem~\ref{thm:istrvg}$(iii)$]
Figure~\ref{fig:circle} shows a representation of a cycle.
\end{proof}

\begin{figure}[h]
	\centering
	\begin{tikzpicture}[scale=0.5]
		\filldraw[red!60!white, draw=black, very thick] (0,0) rectangle (0.5,1);
		\filldraw[red!60!white, draw=black, very thick] (0.5,0.5) rectangle (1,1.5);
		\filldraw[red!60!white, draw=black, very thick] (1,1) rectangle (1.5,2);
		\filldraw[red!60!white, draw=black, very thick] (1.5,1.5) rectangle (2,2.5);
		\draw[dotted, very thick] (2.2,2.3) -- (2.8,2.95);
		\filldraw[red!60!white, draw=black, very thick] (3,2.75) rectangle (3.5,3.75);
		\filldraw[red!60!white, draw=black, very thick] (3.5,3.25) rectangle (4,4.25);
		\filldraw[red!60!white, draw=black, very thick] (3.5,0) rectangle (4,0.5);
	\end{tikzpicture}
	\caption{A representation of a cycle}
	\label{fig:circle}
\end{figure}
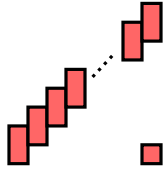

The same idea can be extended to powers of cycles.

\begin{proof}[Proof of Theorem~\ref{thm:poc}$(i)$]
Figure~\ref{fig:c^a_n} shows a representation of a power of a cycle $C^a_n$ where $1\le a\le \frac{n-1}{2}$.
For $a>\frac{n-1}{2}$, $C^a_n$ is a complete graph.
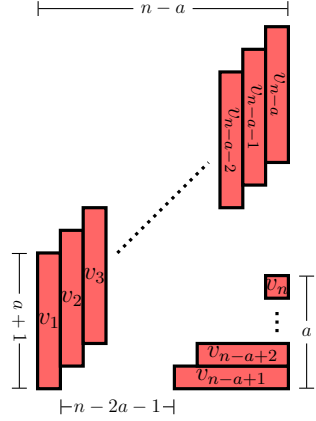
\begin{figure}[h]
	\centering
	\begin{tikzpicture}[scale=0.6]
		\filldraw[red!60!white, draw=black, very thick](0,0)rectangle(0.5,3);
		\draw (0.25,1.5)node[scale=0.8]{$v_1$};
		\filldraw[red!60!white, draw=black, very thick](0.5,0.5)rectangle(1,3.5);
		\draw (0.75,2)node[scale=0.8]{$v_2$};
		\filldraw[red!60!white, draw=black, very thick](1,1)rectangle(1.5,4);
		\draw (1.25,2.5)node[scale=0.8]{$v_3$};
		\draw[dotted, very thick] (1.75,3) -- (3.75,5);
		\filldraw[red!60!white, draw=black, very thick](4,4)rectangle(4.5,7);
		\draw (4.25,5.5)node[rotate=-90,scale=0.75]{$v_{n-a-2}$};
		\filldraw[red!60!white, draw=black, very thick](4.5,4.5)rectangle(5,7.5);
		\draw (4.75,6)node[rotate=-90,scale=0.75]{$v_{n-a-1}$};
		\filldraw[red!60!white, draw=black, very thick](5,5)rectangle(5.5,8);
		\draw (5.25,6.5)node[rotate=-90,scale=0.75]{$v_{n-a}$};
		
		\filldraw[red!60!white, draw=black, very thick](3,0)rectangle(5.5,0.5);
		\draw (4.25,0.25)node[scale=0.8]{$v_{n-a+1}$};
		\filldraw[red!60!white, draw=black, very thick](3.5,0.5)rectangle(5.5,1);
		\draw (4.5,0.75)node[scale=0.8]{$v_{n-a+2}$};
		\draw[dotted, very thick] (5.25,1.25) -- (5.25,1.75);
		\filldraw[red!60!white, draw=black, very thick](5,2)rectangle(5.5,2.5);
		\draw (5.25,2.25)node[scale=0.8]{$v_n$};
		\draw (-0.4,1.5)node[rotate=-90,scale=0.6]{$a+1$};
		\draw[|-](-0.4,3)--(-0.4,2.1);
		\draw[|-](-0.4,0)--(-0.4,1);
		\draw (5.9,1.25)node[scale=0.6]{$a$};
		\draw[|-](5.9,2.5)--(5.9,1.5);
		\draw[|-](5.9,0)--(5.9,1);
		\draw (1.75,-0.4)node[scale=0.6]{$n-2a-1$};
		\draw[|-](0.5,-0.4)--(0.75,-0.4);
		\draw[|-](3,-0.4)--(2.75,-0.4);
		\draw (2.75,8.4)node[scale=0.6]{$n-a$};
		\draw[|-](0,8.4)--(2.1,8.4);
		\draw[|-](5.5,8.4)--(3.4,8.4);
	\end{tikzpicture}
	\caption{A representation of $C^a_n$}
	\label{fig:c^a_n}
\end{figure}
\end{proof}
Now we focus on the complement of a power of a cycle.

\begin{proof}[Proof of Theorem~\ref{thm:poc}$(ii)$]
If $n$ is odd then $D^1_n$ is isomorphic to $C^{\frac{n-3}{2}}_n$ which is a TRVG by Theorem~\ref{thm:poc}$(i)$.
Suppose $n$ is even.
Let $v_1, v_2, \dots, v_n$ be the vertices of $D^1_n$ as in the definition.
First since $v_1$ is not adjacent to $v_2$, we place the rectangles for $v_1$ and $v_2$ not seeing each other (see Figure~\ref{fig:k1reg}).
Since $v_3$ is adjacent to $v_1$ but not $v_2$, we place the rectangle for $v_3$ inside the vision of the rectangle for $v_1$ but outside that for $v_2$.

Continuing in this fashion, for $i \in \{3,4,\dots,n-1\}$, we place the rectangle for $v_i$ inside the vision of the bounding box of $v_1, v_2, \dots, v_{i-2}$ but outside that for $v_{i-1}$.
We place the rectangle for $v_n$ similarly but we make sure that it does not see that for $v_1$.
Thus, $D^1_n$ is a TRVG.
\begin{figure}[h]
	\centering
	\begin{tikzpicture}[scale=0.75]
		\filldraw[red!60!white, draw=black, very thick] (0,0) rectangle (0.5,0.5);
		\draw (0.25,0.25) node[rectangle] {$v_1$};
		\filldraw[red!60!white, draw=black, very thick] (0.5,0.5) rectangle (1,1);
		\draw (0.75,0.75) node[rectangle] {$v_2$};
		\filldraw[red!60!white, draw=black, very thick] (1,0) rectangle (1.5,0.5);
		\draw (1.25,0.25) node[rectangle] {$v_3$};
		\filldraw[red!60!white, draw=black, very thick] (0,1) rectangle (1,1.5);
		\draw (0.5,1.25) node[rectangle] {$v_4$};
		\filldraw[red!60!white, draw=black, very thick] (1.5,0) rectangle (2,1);
		\draw (1.75,0.5) node[rectangle] {$v_5$};
		\filldraw[red!60!white, draw=black, very thick] (0,1.5) rectangle (1.5,2);
		\draw (0.75,1.75) node[rectangle] {$v_6$};
		\draw[dotted, very thick] (0.5,2.2) -- (0.5,2.8);
		\draw[dotted, very thick] (2.2,0.5) -- (2.8,0.5);
		\draw[dotted, very thick] (1.8,1.8) -- (2.8,2.8);
		\filldraw[red!60!white, draw=black, very thick] (0,3) rectangle (3,3.5);
		\draw (1.5,3.25) node[rectangle] {$v_{n-2}$};
		\filldraw[red!60!white, draw=black, very thick] (3,0) rectangle (3.5,2.5);
		\draw (3.25,1.25) node[rectangle, rotate=270] {$v_{n-3}$};
		\filldraw[red!60!white, draw=black, very thick] (0.5,3.5) rectangle (3.5,4);
		\draw (2,3.75) node[rectangle] {$v_n$};
		\filldraw[red!60!white, draw=black, very thick] (3.5,0) rectangle (4,3);
		\draw (3.75,1.5) node[rectangle, rotate=270] {$v_{n-1}$};
	\end{tikzpicture}
	\caption{A representation of $D^1_n$ when $n$ is even}
	\label{fig:k1reg}
\end{figure}
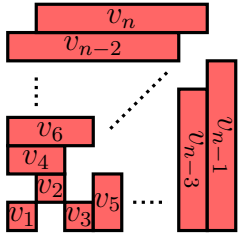
\end{proof}

Next we will prove that the complements of some powers of cycles are non-TRVGs.

\begin{proof}[Proof of Theorem~\ref{thm:poc}$(iii)$]
Consider $D^a_n$ with $n\ge2a+8$ and $a\ge3$.
Let $v_1, v_2, \dots, v_n$ be the vertices of $D^a_n$ as in the definition.
We claim that $K_{4,4}$ is an induced subgraph of $D^a_n$.
Consider the subgraph induced by $\{v_1, v_2, v_3, v_4, v_{a+5}, v_{a+6}, v_{a+7}, v_{a+8}\}$.
Since $a \ge 3$, the sets $\{v_1, v_2, v_3, v_4\}$ and $\{v_{a+5}, v_{a+6}, v_{a+7}, v_{a+8}\}$ are independent.
It remains to show that the indices of $v_4$ and $v_{a+5}$ are more than distance $a$ away and the indices of $v_1$ and $v_{a+8}$ are more than distance $a$ away.
We see that $(a+5)-4>a$ and $(n+1)-(a+8) \ge (2a+9)-(a+8)>a$.
By Theorem~\ref{thm:compbip}, $D^a_n$ is a non-TRVG.
\end{proof}

On the other hand, the complements of some powers of cycles can be TRVGs if $n$ is small compared to $a$.

\begin{proof}[Proof of Theorem~\ref{thm:poc}$(iv)$]
If $n\le 2a+3$ then $D^a_n$ is either an empty graph, a matching or a cycle. These are obviously TRVGs.
Suppose $n=2a+4$.
Let $v_1,v_2,\dots,v_n$ be the vertices of $D^a_n$ as in the definition.
Partition the vertices into four groups, including $V_1=\{v_1,v_2,\dots,v_{a+1}\}$, $V_2=\{v_i:a+3\le i\le n-1\text{ and }i\text{ is odd}\}$, $V_3=\{v_i:a+3\le i\le n-1\text{ and }i\text{ is even}\}$ and $V_4=\{v_{a+2},v_n\}$.
Since $V_i$ is independent for $i\in\{1,2,3\}$, we place the rectangles for its elements diagonally so that they do not see each other.
We choose the position of each group appropriately (see Figures~\ref{fig:d^a_2a+4_1} and~\ref{fig:d^a_2a+4_2}).
Finally, we place the remaining rectangles for elements of $V_4$ as shown.
\end{proof}



\begin{figure}[h]
	\centering
	\begin{tikzpicture}[scale=0.56]
		\draw[step=0.5cm,gray,dashed] (0.1,0.1) grid (12.4,12.9);
		\filldraw[red!60!white, draw=black, very thick] (0.5,0.1) rectangle (1.5,1);
		\draw (1,0.5)node[scale=0.8]{$v_{1}$};
		\filldraw[red!60!white, draw=black, very thick] (1.5,1) rectangle (2.5,2);
		\draw (2,1.5)node[scale=0.8]{$v_{2}$};
		\filldraw[red!60!white, draw=black, very thick] (2.5,2) rectangle (3.5,3);
		\draw (3,2.5)node[scale=0.8]{$v_{3}$};
		\draw[dotted,very thick](3.75,3.25)--(6.25,5.75);
		\filldraw[red!60!white, draw=black, very thick] (6.5,6) rectangle (7.5,7);
		\draw (7,6.5)node[scale=0.8]{$v_{a}$};
		\filldraw[red!60!white, draw=black, very thick] (7.5,7) rectangle (8.5,8);
		\draw (8,7.5)node[scale=0.8]{$v_{a+1}$};
		
		\filldraw[yellow!60!white, draw=black, very thick] (0.1,8) rectangle (1,9);
		\draw (0.55,8.5)node[scale=0.7]{$v_{a+2}$};
		\filldraw[blue!60!white, draw=black, very thick] (1,9) rectangle (2,10);
		\draw (1.5,9.5)node[scale=0.8]{$v_{a+3}$};
		\filldraw[blue!60!white, draw=black, very thick] (2,10) rectangle (4,11);
		\draw (3,10.5)node[scale=0.8]{$v_{a+5}$};
		\draw[dotted,very thick](4.25,11.125)--(5.75,11.875);
		\filldraw[blue!60!white, draw=black, very thick] (6,12) rectangle (8,12.9);
		\draw (7,12.5)node[scale=0.8]{$v_{n-2}$};
		
		\filldraw[green!60!white, draw=black, very thick] (8.5,0.5) rectangle (9.5,2.5);
		\draw (9,1.5)node[scale=0.8]{$v_{a+4}$};
		\filldraw[green!60!white, draw=black, very thick] (9.5,2.5) rectangle (10.5,4.5);
		\draw (10,3.5)node[scale=0.8]{$v_{a+6}$};
		\draw[dotted,very thick](10.625,4.75)--(11.375,6.25);
		\filldraw[green!60!white, draw=black, very thick] (11.5,6.5) rectangle (12.4,8.5);
		\draw (11.95,7.5)node[scale=0.7]{$v_{n-1}$};
		
		\filldraw[yellow!60!white, draw=black, very thick] (8,8.5) rectangle (8.5,9.5);
		\draw (8.25,9)node[scale=0.7]{$v_{n}$};
		\draw[very thick](0.1,0.1)rectangle(12.4,12.9);
	\end{tikzpicture}
	\caption{A representation of $D^a_{2a+4}$ when $a$ is odd}
	\label{fig:d^a_2a+4_1}
\end{figure}
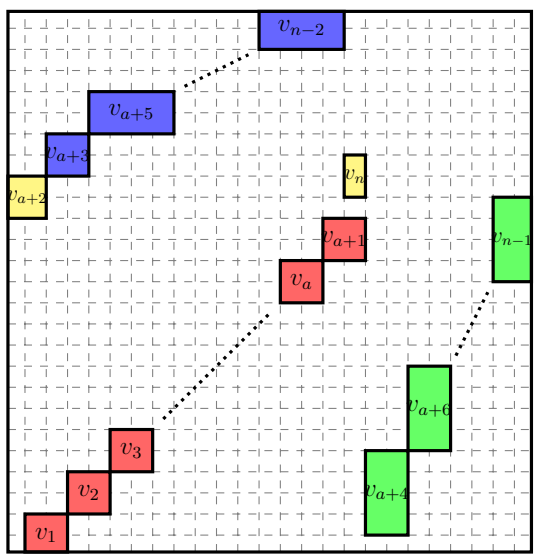
\begin{figure}[h]
	\centering
	\begin{tikzpicture}[scale=0.56]
		\draw[step=0.5cm,gray,dashed] (0.1,0.1) grid (11.4,10.9);
		\filldraw[red!60!white, draw=black, very thick] (0.5,0.1) rectangle (1.5,1);
		\draw (1,0.5)node[scale=0.8]{$v_{1}$};
		\filldraw[red!60!white, draw=black, very thick] (1.5,1) rectangle (2.5,2);
		\draw (2,1.5)node[scale=0.8]{$v_{2}$};
		\filldraw[red!60!white, draw=black, very thick] (2.5,2) rectangle (3.5,3);
		\draw (3,2.5)node[scale=0.8]{$v_{3}$};
		\draw[dotted,very thick](3.75,3.25)--(5.25,4.75);
		\filldraw[red!60!white, draw=black, very thick] (5.5,5) rectangle (6.5,6);
		\draw (6,5.5)node[scale=0.8]{$v_{a}$};
		\filldraw[red!60!white, draw=black, very thick] (6.5,6) rectangle (7.5,7);
		\draw (7,6.5)node[scale=0.8]{$v_{a+1}$};
		
		\filldraw[yellow!60!white, draw=black, very thick] (0.1,7) rectangle (1,8);
		\draw (0.55,7.5)node[scale=0.7]{$v_{a+2}$};
		\filldraw[blue!60!white, draw=black, very thick] (1,8) rectangle (3,9);
		\draw (2,8.5)node[scale=0.8]{$v_{a+4}$};
		\draw[dotted,very thick](3.25,9.125)--(4.75,9.875);
		\filldraw[blue!60!white, draw=black, very thick] (5,10) rectangle (7,10.9);
		\draw (6,10.45)node[scale=0.8]{$v_{n-2}$};
		
		\filldraw[green!60!white, draw=black, very thick] (7.5,0.5) rectangle (8.5,1.5);
		\draw (8,1)node[scale=0.8]{$v_{a+3}$};
		\filldraw[green!60!white, draw=black, very thick] (8.5,1.5) rectangle (9.5,3.5);
		\draw (9,2.5)node[scale=0.8]{$v_{a+5}$};
		\draw[dotted,very thick](9.625,3.75)--(10.375,5.25);
		\filldraw[green!60!white, draw=black, very thick] (10.5,5.5) rectangle (11.4,7.5);
		\draw (10.95,6.5)node[scale=0.7]{$v_{n-1}$};
		
		\filldraw[yellow!60!white, draw=black, very thick] (7,7.5) rectangle (8,8);
		\draw (7.5,7.75)node[scale=0.8]{$v_{n}$};
		\draw[very thick](0.1,0.1)rectangle(11.4,10.9);
	\end{tikzpicture}
	\caption{A representation of $D^a_{2a+4}$ when $a$ is even}
	\label{fig:d^a_2a+4_2}
\end{figure}
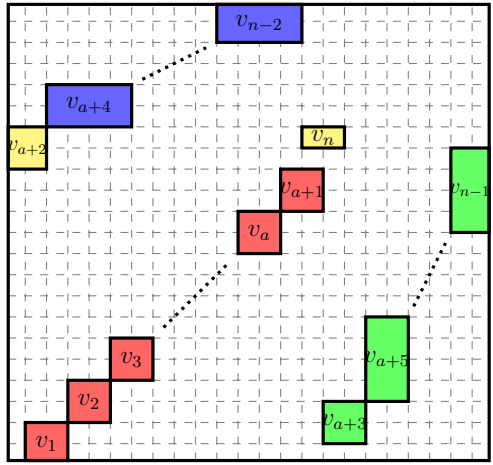

\section{Concluding Remarks}\label{sec:conclude}

In Section~\ref{secgrid}, we show that some natural planar graphs are TRVGs.
However, we cannot find a planar non-TRVG.

\begin{quest}
    Is there a planar non-TRVG?
\end{quest}

In Section~\ref{secbipartite}, we show that $K_{3,5}$ is a non-TRVG. We cannot find a non-TRVG with fewer than $15$ edges.

\begin{conj}
    $K_{3,5}$ is a non-TRVG with the least number of edges.
\end{conj}

In Section~\ref{secpoc}, we show that every $D^a_n$ with $n \ge 2a+8$ vertices is a non-TRVG while every $D^a_n$ with $n \le 2a+4$ vertices is a TRVG for all $a\ge3$. However, we cannot find a representation of the remaining cases.

\begin{conj}
    $D^a_n$ is a non-TRVG when $2a+5\le n\le 2a+7$ and $a\ge 3$.
\end{conj}

What if $a=2$?
Theorem~\ref{thm:poc} shows that $D^2_n$ is a TRVG when $n\le 8$.
Moreover, we found that $D^2_9$ is a TRVG (see Figure~\ref{fig:9k2reg}).
However, we cannot find a representation of $D^2_{10}$ with respect to the plane, but we found a representation of $D^2_{10}$ with respect to the torus (see Figure~\ref{fig:10k2regtorus}).

\begin{conj}
    $D^2_n$ is a non-TRVG when $n\ge10$.
\end{conj}

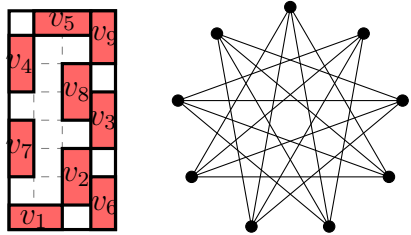
\begin{figure}[h]
	\centering
	\begin{tikzpicture} [scale=0.5]
		\draw[very thick] (-4.4,5.9) rectangle (-1.6,0.1);
		\draw[step=0.75cm,gray,dashed] (-4.4,0.1) grid (-1.6,5.9);
		\filldraw[red!60!white, draw=black, very thick] (-4.4,0.1) rectangle (-3,0.75);
		\draw (-3.75,0.375) node[rectangle] {$v_1$};
		\filldraw[red!60!white, draw=black, very thick] (-2.25,2.25) rectangle (-3,0.75);
		\draw (-2.625,1.5) node[rectangle] {$v_2$};
		\filldraw[red!60!white, draw=black, very thick] (-2.25,2.25) rectangle (-1.6,3.75);
		\draw (-1.925,3) node[rectangle] {$v_3$};
		\filldraw[red!60!white, draw=black, very thick] (-4.4,3.75) rectangle (-3.75,5.25);
		\draw (-4.125,4.5) node[rectangle] {$v_4$};
		\filldraw[red!60!white, draw=black, very thick] (-2.25,5.9) rectangle (-3.75,5.25);
		\draw (-3,5.625) node[rectangle] {$v_5$};
		\filldraw[red!60!white, draw=black, very thick] (-2.25,0.1) rectangle (-1.6,1.5);
		\draw (-1.875,0.75) node[rectangle] {$v_6$};
		\filldraw[red!60!white, draw=black, very thick] (-4.4,1.5) rectangle (-3.75,3);
		\draw (-4.125,2.25) node[rectangle] {$v_7$};
		\filldraw[red!60!white, draw=black, very thick] (-3,3) rectangle (-2.25,4.5);
		\draw (-2.625,3.75) node[rectangle] {$v_8$};
		\filldraw[red!60!white, draw=black, very thick] (-2.25,4.5) rectangle (-1.6,5.9);
		\draw (-1.875,5.25) node[rectangle] {$v_9$};
		\filldraw (1.977,0.18) circle(0.15cm);
		\filldraw (4.023,0.18) circle(0.15cm);
		\filldraw (5.596,1.497) circle(0.15cm);
		\filldraw (5.955,3.52) circle(0.15cm);
		\filldraw (4.93,5.298) circle(0.15cm);
		\filldraw (3,6) circle(0.15cm);
		\filldraw (1.07,5.298) circle(0.15cm);
		\filldraw (0.045,3.52) circle(0.15cm);
		\filldraw (0.404,1.497) circle(0.15cm);
		
		\draw (3,6) -- (0.404,1.497);
		\draw (3,6) -- (1.977,0.18);
		\draw (3,6) -- (4.023,0.18);
		\draw (3,6) -- (5.596,1.497);
		\draw (1.07,5.298) -- (1.977,0.18);
		\draw (1.07,5.298) -- (4.023,0.18);
		\draw (1.07,5.298) -- (5.596,1.497);
		\draw (1.07,5.298) -- (5.955,3.52);
		\draw (0.045,3.52) -- (4.023,0.18);
		\draw (0.045,3.52) -- (5.596,1.497);
		\draw (0.045,3.52) -- (5.955,3.52);
		\draw (0.045,3.52) -- (4.93,5.298);
		\draw (0.404,1.497) -- (5.596,1.497);
		\draw (0.404,1.497) -- (5.955,3.52);
		\draw (0.404,1.497) -- (4.93,5.298);
		\draw (1.977,0.18) -- (5.955,3.52);
		\draw (1.977,0.18) -- (4.93,5.298);
		\draw (4.023,0.18) -- (4.93,5.298);
	\end{tikzpicture}
	\caption{A representation of $D^2_9$}
	\label{fig:9k2reg}
\end{figure}

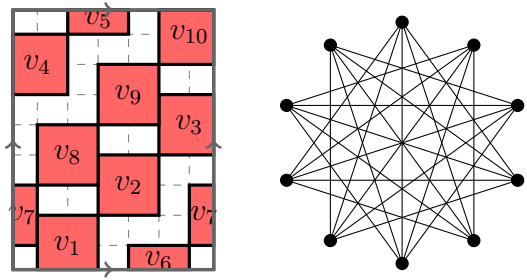
\begin{figure}[h!]
	\centering
	\begin{tikzpicture} [scale=0.8]
		\draw[very thick] (-4.4,4.4) rectangle (-1.1,0.1);
		\draw[step=0.5cm,gray,dashed] (-4.4,4.4) grid (-1.1,0.1);
		\filldraw[red!60!white, draw=black, very thick] (-2.5,0.1) rectangle (-1.5,0.5);
		\draw (-2,0.25) node[rectangle] {$v_6$};
		\filldraw[red!60!white, draw=black, very thick] (-1.5,0.5) rectangle (-1.1,1.5);
		\draw (-1.25,1) node[rectangle] {$v_7$};
            \filldraw[red!60!white, draw=black, very thick] (-4.4,0.5) rectangle (-4,1.5);
		\draw (-4.25,1) node[rectangle] {$v_7$};
		\filldraw[red!60!white, draw=black, very thick] (-4,0.1) rectangle (-3,1);
		\draw (-3.5,0.5) node[rectangle] {$v_1$};
		\filldraw[red!60!white, draw=black, very thick] (-3,1) rectangle (-2,2);
		\draw (-2.5,1.5) node[rectangle] {$v_2$};
		\filldraw[red!60!white, draw=black, very thick] (-2,2) rectangle (-1.1,3);
		\draw (-1.5,2.5) node[rectangle] {$v_3$};
		\filldraw[red!60!white, draw=black, very thick] (-4.4,3) rectangle (-3.5,4);
		\draw (-4,3.5) node[rectangle] {$v_4$};
		\filldraw[red!60!white, draw=black, very thick] (-3.5,4) rectangle (-2.5,4.4);
		\draw (-3,4.25) node[rectangle] {$v_5$};
		\filldraw[red!60!white, draw=black, very thick] (-4,1.5) rectangle (-3,2.5);
		\draw (-3.5,2) node[rectangle] {$v_8$};
		\filldraw[red!60!white, draw=black, very thick] (-3,2.5) rectangle (-2,3.5);
		\draw (-2.5,3) node[rectangle] {$v_9$};
            \filldraw[red!60!white, draw=black, very thick] (-2,3.5) rectangle (-1.1,4.4);
		\draw (-1.5,4) node[rectangle] {$v_{10}$};
        \draw[very thick,->,gray!80!black] (-4.4,0.1)--(-2.75,0.1);
        \draw[very thick,->,gray!80!black] (-1.1,0.1)--(-1.1,2.25);
        \draw[very thick,->,gray!80!black] (-4.4,4.4)--(-2.75,4.4);
        \draw[very thick,->,gray!80!black] (-4.4,0.1)--(-4.4,2.25);
        \draw[very thick,gray!80!black] (-4.4,0.1) rectangle (-1.1,4.4);
        
            \filldraw (2,4.2) circle(0.1cm);
		\filldraw (0.82,3.82) circle(0.1cm);
		\filldraw (0.1,2.82) circle(0.1cm);
		\filldraw (0.1,1.58) circle(0.1cm);
		\filldraw (0.82,0.58) circle(0.1cm);
		\filldraw (2,0.2) circle(0.1cm);
		\filldraw (3.18,0.58) circle(0.1cm);
		\filldraw (3.9,1.58) circle(0.1cm);
		\filldraw (3.9,2.82) circle(0.1cm);
		\filldraw (3.18,3.82) circle(0.1cm);
		
	\draw (2,4.2) -- (0.1,1.58);
	\draw (2,4.2) -- (0.82,0.58);
        \draw (2,4.2) -- (2,0.2);
        \draw (2,4.2) -- (3.18,0.58);
        \draw (2,4.2) -- (3.9,1.58);
        \draw (0.82,3.82) -- (0.82,0.58);
        \draw (0.82,3.82) -- (2,0.2);
        \draw (0.82,3.82) -- (3.18,0.58);
        \draw (0.82,3.82) -- (3.9,1.58);
        \draw (0.82,3.82) -- (3.9,2.82);
        \draw (0.1,2.82) -- (2,0.2);
        \draw (0.1,2.82) -- (3.18,0.58);
        \draw (0.1,2.82) -- (3.9,1.58);
        \draw (0.1,2.82) -- (3.9,2.82);
        \draw (0.1,2.82) -- (3.18,3.82);
        \draw (0.1,1.58) -- (3.18,0.58);
        \draw (0.1,1.58) -- (3.9,1.58);
        \draw (0.1,1.58) -- (3.9,2.82);
        \draw (0.1,1.58) -- (3.18,3.82);
        \draw (0.82,0.58) -- (3.9,1.58);
        \draw (0.82,0.58) -- (3.9,2.82);
        \draw (0.82,0.58) -- (3.18,3.82);
        \draw (2,0.2) -- (3.9,2.82);
        \draw (2,0.2) -- (3.18,3.82);
        \draw (3.18,0.58) -- (3.18,3.82);
	\end{tikzpicture}
	\caption{A representation of $D^2_{10}$ with respect to torus}
	\label{fig:10k2regtorus}
\end{figure}

\bibliographystyle{siam}
\bibliography{TRVGs}

\end{document}